\documentclass[12pt]{elsarticle}

\usepackage{mystyleElsArticle}
\usepackage{mathrsfs}
\usepackage{bm}
\usepackage{booktabs}
\usepackage{color}
\usepackage{epstopdf}
\newcommand{\Hc}{\mathcal{H}}

\newcommand{\oT}{\mathcal{T}}
\newcommand{\oN}{\mathcal{N}}

\DeclareMathOperator*{\esssup}{ess\,sup}

\newtheorem{thm}{Theorem}[section]
\newtheorem{lem}{Lemma}[section]

\newtheorem{asp}{Assumption}[section]

\newtheorem{rem}{Remark}[section]

\numberwithin{equation}{section}

\makeatletter
\def\ps@pprintTitle{%
  \let\@oddhead\@empty
  \let\@evenhead\@empty
  \let\@oddfoot\@empty
  \let\@evenfoot\@oddfoot
}
\makeatother
\usepackage[margin=0.95in]{geometry}    

\usepackage{lineno}

\begin{document}
\begin{frontmatter}
		\title{Convergence analysis of an operator-compressed multiscale finite element method for Schr\"{o}dinger equations with multiscale potentials}
								
		\author[hku]{Zhizhang Wu}
		\ead{wuzz@hku.hk}
		\author[hku]{Zhiwen Zhang\corref{cor1}}
		\ead{zhangzw@hku.hk}
		
		\address[hku]{Department of Mathematics, The University of Hong Kong, Pokfulam Road, Hong Kong SAR, China.}
		\cortext[cor1]{Corresponding author}
\begin{abstract}
\noindent
In this paper, we analyze the convergence of the operator-compressed multiscale finite element method (OC MsFEM) for Schr\"{o}dinger equations with general multiscale potentials in the semiclassical regime. In the OC MsFEM the multiscale basis functions are constructed by solving a constrained energy minimization. Under a mild assumption on the mesh size $H$, we prove the exponential decay of the multiscale basis functions so that localized multiscale basis functions can be constructed, which achieve the same accuracy as the global ones if the oversampling size $m = O(\log(1/H))$. We prove the first-order convergence in the energy norm and second-order convergence in the $L^2$ norm for the OC MsFEM and super convergence rates can be obtained if the solution possesses sufficiently high regularity.  By analysing the regularity of the solution, we also derive the dependence of the error estimates on the small parameters of the Schr\"{o}dinger equation. We find that the OC MsFEM outperforms the finite element method (FEM) due to the super convergence behavior for high-regularity solutions and weaker dependence on the small parameters for low-regularity solutions in the presence of the multiscale potential. Finally, we present numerical results to demonstrate the accuracy and robustness of the OC MsFEM.

\medskip
\noindent{\textbf{Key words.}} Schr\"{o}dinger equation;  semiclassical regime; multiscale potential; multiscale finite element basis; operator compression;  convergence analysis.
			
\medskip
\noindent{{\textbf{AMS subject classifications.}} 65M12, 65M60, 65K10, 35Q41, 74Q10}
\end{abstract}	
\end{frontmatter}

\section{Introduction}
\label{sec: introduction}

In solid state physics, an important model to describe the motion of an electron in the medium with microstructures is the Schr\"{o}dinger equation with a multiscale potential in the semiclassical regime. A widely studied model is the electron motion in a perfect crystal with an external field, where the potential is a linear combination of an oscillatory periodic potential and a slowly-varying external potential. This model can be efficiently solved by a number of numerical schemes that make use of the periodic structure of the potential, e.g. the Bloch decomposition-based time-splitting pseudospectral method \cite{Huangetal:2007, Huang:2008,WuHuang:2016}, the Gaussian beam method \cite{Jin:08, Jin:2010, Qian:2010, Yin:2011}, and the frozen Gaussian approximation method \cite{DelgadilloLuYang:2016}. With the recent development in nanotechnology, increasing interest has been shown in quantum heterostructures with tailored functionalities, such as heterojunctions, including the ferromagnet/metal/ferromagnet structure for giant megnetoresistance \cite{Zutic:2004}, the silicon-based heterojunction for solar cells \cite{Louwenetal:2016}, and quantum metamaterials \cite{Quach:2011}. For the electron motion in these heterostructures, however, the potential cannot be formulated in the above-mentioned form since a basic feature of these devices is the combination of dissimilar crystalline structures, which leads to a heterogeneous interaction of the electron with ionic cores in different lattice structures. Consequently, available methods based on asymptotic analysis \cite{miller2006applied,murray2012asymptotic} cannot be applied to these heterogeneous models since these methods require an additive form of different scales in the potential term in order to construct the prescribed approximate solutions.

In this paper, we consider the semiclassical Schr\"{o}dinger equation with a general multiscale potential
\begin{equation} \label{eq: time-dependent Schrodinger equation in the whole space}
\left\{
\begin{aligned}
& i \varepsilon \partial_t u^{\varepsilon, \delta} = -\frac{1}{2} \varepsilon^2 \Delta u^{\varepsilon, \delta} + V^{\delta}(\bx) u^{\varepsilon, \delta}, \quad \bx=(x_1,...,x_d) \in \mathbb{R}^d, \quad t \in \mathbb{R}, \\
& u^{\varepsilon, \delta}|_{t = 0} = u_0^{\varepsilon}(\bx), \quad \bx \in \mathbb{R}^d,
\end{aligned}
\right.
\end{equation}
where $0 < \varepsilon \ll 1$ is an effective Planck constant describing the microscopic/macroscopic scale ratio, $u_0^{\varepsilon}(\bx)$ is the initial data that is dependent on the effective Planck constant $\varepsilon$ (motivated by the WKB approximation), $d$ is the spatial dimension, and $V^{\delta}(\bx) \in \mathbb{R}$ is a general multiscale potential depending on the small parameter $0 < \delta \ll 1$, which describes the multiscale structure in the potential and can be different from $\varepsilon$. By assuming the dependence of the potential on $\delta$, we can see in the subsequent analysis how the multiscale potential $V^{\delta}$ influences the regularity of the solution and hence the error estimates of different numerical methods.

For the numerical solution of \eqref{eq: time-dependent Schrodinger equation in the whole space}, traditional methods like the FEM \cite{dorfler1996time} and finite difference method (FDM) \cite{markowich1999numerical,markowich2002wigner} are prohibitively costly due to the strong mesh size restrictions induced by the small effective Planck constant and multiscale structures in the potential, while the time-splitting spectral method \cite{bao2002time} would suffer from reduced convergence order and great approximation errors if the potential possesses discontinuities. In order to efficiently compute \eqref{eq: time-dependent Schrodinger equation in the whole space} with the multiscale potential $V^{\delta}(\bx)$ in a general form, an OC MsFEM for the Schr\"{o}dinger equation was proposed in \cite{chen2019multiscale}. The OC MsFEM for the Schr\"{o}dinger equation is motivated by several works relevant to the compression of the elliptic operator with heterogeneous and highly varying coefficients, e.g. the multigrid method for multiscale problems from the perspective of a decision theory discussed in \cite{Owhadi:2015,Owhadi:2017}, the sparse operator compression of high-order elliptic operators with rough coefficients studied in \cite{hou2017sparse} and the modified variational multiscale method using correctors introduced in \cite{Malqvist:2014}. And we remark here that many efficient methods have also been developed for the multiscale PDEs in the past few decades. See for example \cite{chung2018constraint,EngquistE:03,EfendievHou:09,HanZhangCMS:12,HouWuCai:99,Hughes:1998,Jenny:03,Kevrekidis:2003,OwhadiZhang:07} and the references therein.

In the OC MsFEM for the Schr\"{o}dinger equation, the multiscale basis functions are constructed via a constrained energy minimization associated with the hamiltonian $\Hc = - \frac{1}{2} \varepsilon^2 \Delta + V^{\delta}(\bx)$. The fully discrete scheme can be given with a finite difference scheme in temporal discretization, e.g. the Crank-Nicolson scheme. Through the energy minimization, the local microstructures induced by the hamiltonian $\Hc$ are incorporated in the basis functions so that the multiscale features of the solution are well captured by the basis functions. Moreover, the energy minimization can be solved numerically for arbitrary bounded multiscale potentials and thus the OC MsFEM for the Schr\"{o}dinger equation can be applied to a multiscale potential $V^{\delta}$ in a general form. In \cite{chen2019multiscale}, the OC MsFEM is shown to be accurate for various types of multiscale potentials. So far, however, there have been no rigorous results on the approximation error of the OC MsFEM for the semiclassical Schr\"{o}dinger equation with multiscale potentials.

In this paper, we focus on the convergence analysis of the OC MsFEM for Schr\"{o}dinger equations with multiscale potentials in the semiclassical regime. The property of exponential decay is proved for the multiscale basis functions constructed through the constrained energy minimization, provided that the mesh size $H = O(\varepsilon)$. Thus, the localized multiscale basis functions can be constructed via a modified constrained energy minimization. The localized basis functions are shown to admit the same accuracy as the global ones if the oversampling size $m = O(\log(1/H))$. By using the properties of Cl\'{e}ment-type interpolation \cite{brenner2007mathematical,clement1975approximation,scott1990finite}, convergence rates of first order in the energy norm and second order in $L^2$ norm are proved for the Galerkin approximation in the multiscale finite element space. Furthermore, super convergence rates of second order in the energy norm and third order in $L^2$ norm can be achieved if the solution possesses sufficiently high regularity. Combining the analysis on the regularity of the solution, we also derive the dependence of the error estimates on the small parameters $\varepsilon$ and $\delta$. We find that using the same mesh size the OC MsFEM gives more accurate results than the FEM for the  Schr\"{o}dinger equation with multiscale potentials due to its super convergence behavior for high-regularity solutions and weaker dependence on the small parameters $\varepsilon$ and $\delta$ for low-regularity solutions. The weaker dependence of the OC MsFEM on the small parameters results from the fact that the projection error estimate of OC MsFEM depends on higher temporal regularity of the solution while the projection error estimate of FEM depends on higher spatial regularity of solution, and the fact that the spatial derivatives of the solution are more oscillatory than the time derivatives in the presence of the multiscale potential; see Lemmas \ref{lem: temporal regularity}, \ref{lem: spatial regularity}, \ref{lem: temporal spatial regularity}, Remark \ref{rem: more oscillatory spatial derivatives} and the discussion in Section \ref{sec: Projection error}. Finally, we present numerical results to confirm our theoretical findings. The numerical examples also show that the OC MsFEM is robust in the sense that it still yields high accuracy for discontinuous multiscale potentials.

The rest of the paper is organized as follows. In Section \ref{sec: preliminaries}, we will introduce the setting of the problem and some preliminaries on the Cl\'{e}ment-type interpolation, and we will also prove some estimates of the regularity of the solution. In Section \ref{sec: OC basis}, we will prove the exponential decay of the global basis functions and discuss the approximation property of the projection in both global and localized multiscale space. Then, we will study the convergence rates of the OC MsFEM for the Schr\"{o}dinger equation in Section \ref{sec: convergence analysis}. Numerical examples will be shown in Section \ref{sec: numerical experiments} to support our analysis. Finally, some conclusions will be drawn in Section \ref{sec: conclusion}.

\section{Problem setting and some preparations}
\label{sec: preliminaries}
In this section, the problem setting of the semiclassical Schr\"{o}dinger equation with a multiscale potential is formulated. Then some results on the Cl\'{e}ment-type interpolation are introduced. In addition, we prove some estimates of the regularity of the solution.

All functions are complex-valued and the conjugate of a function $v$ is denoted by $\bar{v}$. Standard notations on Sobolev space are used. The spatial derivative is denoted by $D_\bx^{\bm{\sigma}}$, where $D_\bx^{\bm{\sigma}} w = \partial_{x_1}^{\sigma_1}\cdots\partial_{x_d}^{\sigma_d} w$ with the multi-index $\bm{\sigma} = (\sigma_1, \ldots, \sigma_d) \in \mathbb{N}^d$ and $|\bm{\sigma}| = \sigma_1 + \cdots + \sigma_d$. The spatial $L^2$ inner product is denoted by $(\cdot, \cdot)$ with $(v, w) = \int_{\Omega} v \bar{w}$, the spatial $L^2$ norm is denoted by $|| \cdot ||$ with $|| w ||^2 = (w, w)$, $|| \cdot ||_{\infty}$ is the spatial $L^{\infty}$ norm with $|| w ||_{\infty} = \esssup_{\bx \in \Omega} | w(\bx) |$ and the spatial $H^k$ norm is denoted by $|| \cdot ||_{H^k}$ with $|| w ||_{H^k}^2 = || w ||^2 + \sum_{0 < |\bm{\sigma}| \le k} || D_\bx^{\bm{\sigma}} w ||^2$. And we define $H^1_P(\Omega) = \{ w \in H^1(\Omega) | w \text{ is periodic on } \partial \Omega \}$, where $\Omega$ is a bounded domain. To simplify notations, we denote by $C$ a generic positive constant which may be different at each occurrence but is independent of the small parameters $\varepsilon, \delta$, the oversampling size $m$, the spatial mesh size $H$, and the time step size $\Delta t$.

\subsection{Model setting}
For numerical purposes, \eqref{eq: time-dependent Schrodinger equation in the whole space} is restricted on a bounded domain $\Omega = [0, 2\pi]^d$ with prescribed periodic boundary conditions. The following problem is considered:
\begin{equation} \label{eq: time-dependent Schrodinger equation}
\left\{
\begin{aligned}
& i \varepsilon \partial_t u^{\varepsilon, \delta} = -\frac{1}{2} \varepsilon^2 \Delta u^{\varepsilon, \delta} + V^{\delta}(\bx) u^{\varepsilon, \delta}, \quad \bx \in \Omega, 0 < t \le T, \\
& u^{\varepsilon, \delta}, D^{\bm{\sigma}}_\bx u^{\varepsilon, \delta} \text{ are periodic on } \partial \Omega, \quad |\bm{\sigma}| = 1, 0 < t \le T, \\
& u^{\varepsilon, \delta}|_{t = 0} = u_0^{\varepsilon}(\bx), \bx \in \Omega.
\end{aligned}
\right.
\end{equation}
We assume that $u_0^{\varepsilon}(\bx)$ satisfies $|| D_\bx^{\bm{\sigma}} u_0^{\varepsilon} || \le \frac{C}{\varepsilon^{|\bm{\sigma}|}}$. And for the multiscale potential $V^{\delta}$, we assume that $V_{\min} \le V^{\delta}(\bx) \le V_{\max}$, $\forall \bx \in \Omega$, where $0 < V_{\min} \le V_{\max}$ and $|| D^{\bm{\sigma}}_\bx V^{\delta} ||_{\infty} \le \frac{C}{\delta^{|\bm{\sigma}|}}$. Moreover, we assume that $D^{\bm{\sigma}}_\bx u^{\varepsilon, \delta}$ are periodic on $\partial \Omega$ for $|\bm{\sigma}| = 2, 3$ and $0 < t \le T$.

\begin{rem}
If $\tilde{u}^{\varepsilon, \delta}$ is the solution of \eqref{eq: time-dependent Schrodinger equation} with the potential $\tilde{V}^{\delta}$, where $-V_0 \le \tilde{V}^{\delta} \le V_0$ for some $V_0 > 0$, we may set $u^{\varepsilon, \delta} = e^{- 2 i V_0 t / \varepsilon} \tilde{u}^{\varepsilon, \delta}$. Then, $u^{\varepsilon, \delta}$ is the solution of \eqref{eq: time-dependent Schrodinger equation} with the potential $V^{\delta} = \tilde{V}^{\delta} + 2V_0$ and $V_0 \le V^{\delta} \le 3V_0$.
\end{rem}

In the following, for brevity of notations, the superscripts $\varepsilon, \delta$ will be dropped for $u_0^{\varepsilon}, u^{\varepsilon, \delta}$ and $V^{\delta}$ unless necessary. We introduce the bilinear form associated with the Schr\"{o}dinger operator $\Hc = -\frac{1}{2} \varepsilon^2 \Delta + V$ as
\begin{equation}
a(v, w) = \frac{1}{2} \varepsilon^2 (\nabla v, \nabla w) + (Vv,w).
\end{equation}
The following energy norm is introduced:
\begin{equation}
|| w ||_{e} = a(w, w)^{\frac{1}{2}} = \left( \frac{\varepsilon^2}{2} || \nabla w ||^2 + (Vw, w) \right)^{\frac{1}{2}}.
\end{equation}
Then, the energy norm $|| \cdot ||_e$ is equivalent to the $H^1$ norm $|| \cdot ||_{H^1}$ and it is easy to prove the following lemma.
\begin{lem}
For any $v, w \in H^1(\Omega)$,
\begin{equation}
|a(v, w)| \le || v ||_e || w ||_e.
\end{equation}
\end{lem}

If the stationary problem with $\Hc$ as the differential operator
\begin{equation} \label{eq: steady Schrodinger equation}
\left\{
\begin{aligned}
& \Hc u = f, \quad \bx \in \Omega, \\
& u, D^{\bm{\sigma}}_\bx u \text{ are periodic on } \partial \Omega, |\bm{\sigma}| = 1,
\end{aligned}
\right.
\end{equation}
is considered, where periodic boundary conditions are prescribed and $f \in L^2(\Omega)$, the associated variational problem would be to find $u \in H^1_P(\Omega)$ such that
\begin{equation} \label{eq: variational form of steady Schrodinger}
a(u, v) = (f, v), \quad \forall v \in H^1_P(\Omega).
\end{equation}
By the Lax-Milgram theorem, the variational problem \eqref{eq: variational form of steady Schrodinger} admits a unique solution $u \in H^1_P(\Omega)$ with a stability estimate
\begin{equation}
|| u ||_e \le C_{\text{st}}(\varepsilon, V) || f ||.
\end{equation}

\subsection{Cl\'{e}ment-type interpolation}

Let $\oT_H = \{ T_e \}_{e = 1}^{N_e}$ be some quasi-uniform and shape-regular simplicial finite element meshes \cite{ciarlet2002finite,hackbusch2015hierarchical,hackbusch2017elliptic} of $\Omega$ with mesh size $H$, where $N_e$ is the number of elements. Then for $K$ being the union of some elements in $\oT_H$, the neighbourhood of $K$ can be defined as
\begin{equation}
N(K) = \bigcup_{G \in \oT_H, G \cap K \neq \emptyset} G.
\end{equation}
And for $m \in \mathbb{N}$, $N^{m + 1}(K) = N(N^m(K))$, where $N^0(K) = K$, and $m$ is referred to as the oversampling size. If we define
\begin{equation}
\eta(\bx) = \frac{\text{dist}(\bx, N^{m}(K))}  { \text{dist}(\bx, N^{m}(K)) + \text{dist}(\bx, \Omega \backslash N^{m + 1}(K)) }
\end{equation}
for some $m \in \mathbb{N}$, the shape regularity of $\oT_H$ implies that $H || \nabla \eta ||_{\infty} \le \gamma$, where $\gamma$ is independent of $\varepsilon$, $\delta$, $m$ and $H$. The shape regularity and quasi-uniformness also imply that there exists a constant $C_{\text{ol}}$ independent of $\varepsilon$, $\delta$, $m$ and $H$ \cite{peterseim2017eliminating,peterseim2020computational} such that
\begin{equation} \label{eq: bounded number of overlapping element}
\max \limits_{T \in \oT_H} \text{card} \{ G \in \oT_H | G \subset N(T) \} \le C_{\text{ol}}.
\end{equation}

The first-order conforming finite element space of $\oT_H$ is given by
\begin{equation}
\Phi_H = \{ \phi \in H_P^1(\Omega)~|~\forall T \in \oT_H, \phi|_{T} \text{ is a polynomial of total degree } \le 1 \}.
\end{equation}
Let $\oN_H$ be the set of vertices of $\oT_H$ with repeated vertices due to the periodic boundary conditions removed and $N_H = |\oN_H|$. Then $\Phi_H = \text{span}\{ \phi_j, j = 1, \ldots, N_H \}$, where $\phi_j \in \Phi_H, j = 1, \ldots, N_H$ is the nodal basis satisfying
$\phi_j(\bx_k) = \delta_{jk}$,  $\forall \bx_k \in \oN_H$.
The Cl\'{e}ment-type interpolation operator $I_H$ \cite{brenner2007mathematical,clement1975approximation,scott1990finite} is defined by
\begin{equation}
I_H v = \sum_{j = 1}^{N_H} \alpha_j(v) \phi_j, \quad \forall v \in H_P^1(\Omega),
\end{equation}
where $\alpha_j(v) = \frac{(v, \phi_j)}{(1, \phi_j)}$. Then, the local approximation and stability properties of the interpolation operator $I_H$ \cite{carstensen1999edge} guarantee that there exists a constant $C_{I_H}$ only depending on the shape regularity such that
\begin{equation} \label{eq: local approximation of Clement-type interpolation}
H^{-1} || v - I_H v ||_T + || \nabla( v - I_H v ) ||_T \le C_{I_H} || \nabla v ||_{N(T)}, \quad \forall T \in \oT_H,
\end{equation}
where $ (v , w)_K = \int_K v \bar{w} $ and $|| v ||_K^2 = (v, v)_K$ denote the spatial $L^2$ inner product and spatial $L^2$ norm restricted on $K \subset \Omega$, respectively. Set $W = \ker(I_H)$. Then $H_P^1(\Omega) = \Phi_H \oplus W$ and $(v, w) = 0$, $\forall v \in \Phi_H, w \in W$. We have the following lemma.

\begin{lem} \label{lem: property of clement-type interpolation}
For $v \in W, f \in L^2(\Omega)$, there holds
\begin{equation} \label{eq: convergence of Clement-type interpolation}
(f, v) \le C H || f || || \nabla v ||.
\end{equation}
Moreover, if $f \in H^1(\Omega)$, then
\begin{equation} \label{eq: super convergence of Clement-type interpolation}
(f, v) \le C H^2 || \nabla f || || \nabla v ||.
\end{equation}
\end{lem}

\begin{proof}
Since $v \in W$, then $I_H v = 0$. By \eqref{eq: bounded number of overlapping element} and \eqref{eq: local approximation of Clement-type interpolation}, we have
\begin{equation}
(f, v) = (f, v - I_H v) \le || f || ||v - I_H v|| \le C H || f || || \nabla v ||.
\end{equation}
Note that $(I_H f, v) = 0$. Hence if we further have $f \in H^1(\Omega)$, then
\begin{equation}
(f, v) = (f - I_H f, v - I_H v) \le C H^2 || \nabla f || || \nabla v ||.
\end{equation}
\end{proof}

Moreover, there exists a local right inverse of $I_H$ \cite{henning2015multiscale}, denoted by $I_H^{-1, \text{loc}}: \Phi_H \rightarrow H_P^1(\Omega)$, satisfying
\begin{equation}
I_H(I_H^{-1, \text{loc}} v_H) = v_H,
\end{equation}
\begin{equation}
|| \nabla I_H^{-1, \text{loc}} v_H || \le C_{I_H}^{\prime} || \nabla v_H ||,
\end{equation}
\begin{equation}
\text{supp}(I_H^{-1, \text{loc}} v_H) = \bigcup \{ T \in \oT_H | T \cap \overline{ \text{supp}(v_H) } \neq \emptyset \},
\end{equation}
where $v_H \in \Phi_H$ and $C_{I_H}'$ only depends on the shape regularity.

\subsection{Regularity of the solutions of Schr\"{o}dinger equations with multiscale potentials}

We are interested in studying the dependence of the error estimates of OC MsFEM on the small parameters $\varepsilon, \delta$. Hence we need to study the temporal and spatial regularity of the solution $u$ of the Schr\"{o}dinger equation \eqref{eq: time-dependent Schrodinger equation}. We first study the temporal regularity of $u$.

\begin{lem} \label{lem: temporal regularity}
If $\partial_t^k u(t) \in L^2(\Omega)$ for any $t \in [0, T]$, where $k = 1, 2, 3, 4$, then it holds true that for any $0 \le t \le T$
\begin{equation}
|| \partial_t^k u(t) || \le \frac{C \varepsilon^{k - 2}}{\min\{ \varepsilon^{2k - 2}, \delta^{2k - 2} \}}.
\label{partialktu}
\end{equation}
\end{lem}

\begin{proof}
For $u_t$, taking the time derivative of \eqref{eq: time-dependent Schrodinger equation}, multiplying it with $\bar{u}_t$, integrating it w.r.t. $\bx$ over $\Omega$ and taking the imaginary part, we have
$\frac{d}{dt} || u_t ||^2 = 0$, which implies $|| u_t(t) || = || u_t(0) ||, \forall t \in [0, T]$. For $|| u_t(0) ||$, we have $i \varepsilon u_t(0) = - \frac{\varepsilon^2}{2} \Delta u_0 + V u_0$,
which indicates
\begin{equation}
|| u_t(0) || \le \frac{\varepsilon}{2} || \Delta u_0 || + \frac{1}{\varepsilon} || V u_0 || \le \frac{C}{\varepsilon}.
\end{equation}
Applying similar procedures to $u_{tt}, u_{ttt}$ and $u_{tttt}$, we can obtain the results \eqref{partialktu}.
\end{proof}

Then, we turn to the spatial regularity of $u$.
\begin{lem} \label{lem: spatial regularity}
If $D_\bx^{\bm{\sigma}} u(t) \in L^2(\Omega)$ for any $t \in [0, T]$ and $|\bm{\sigma}| = k$, where $k = 1, 2$, then it holds true that for any $0 \le t \le T$
\begin{equation}
|| D_\bx^{\bm{\sigma}} u(t) || \le \frac{C}{\varepsilon^k \delta^k}.
\label{Dxut}
\end{equation}
\end{lem}

\begin{proof}
For $| \bm{\sigma} | = 1$, it is sufficient to prove \eqref{Dxut} for $u_{x_1}$. We have $|| u(t) || = || u_0 ||$ for any $t \in [0, T]$, which is the conservation of mass. Taking the spatial partial derivative of \eqref{eq: time-dependent Schrodinger equation} w.r.t. $x_1$, multiplying it with $\bar{u}_{x_1}$, integrating it w.r.t. $\bx$ over $\Omega$ and taking the imaginary part, we have
\begin{equation}
\varepsilon \frac{1}{2} \frac{d}{dt} || u_{x_1} ||^2 = \text{Im}( V_{x_1} u, u_{x_1} ).
\end{equation}
Then
\begin{equation}
\frac{1}{2} \frac{d}{dt} || u_{x_1} ||^2 = || u_{x_1} || \frac{d}{dt} || u_{x_1} || \le \frac{1}{\varepsilon} |( V_{x_1} u, u_{x_1} )| \le \frac{1}{\varepsilon \delta} || u || || u_{x_1} ||,
\end{equation}
and hence $\frac{d}{dt} || u_{x_1} || \le \frac{1}{\varepsilon \delta} || u ||$. Therefore, we have
\begin{equation}
|| u_{x_1}(t) || \le || \partial_{x_1} u_0 || + \frac{T}{\varepsilon \delta}|| u_0 || \le \frac{C}{\varepsilon \delta}, \quad \forall t \in [0, T],
\end{equation}
which implies \eqref{Dxut} for $| \bm{\sigma} | = 1$. Then, applying the same procedure to any $D^{\bm{\sigma}}_\bx u$ with $|\bm{\sigma}| = 2$, we obtain the result.
\end{proof}

Furthermore, we have
\begin{lem} \label{lem: temporal spatial regularity}
If $\partial_t^k u(t) \in H^1(\Omega)$ for any $t \in [0, T]$, where $k = 1, 2, 3$, then it holds true that for any $0 \le t \le T$
\begin{equation}
|| \nabla \partial_t^k u(t) || \le \frac{C \varepsilon^{k - 2}}{ \varepsilon \delta \min\{ \varepsilon^{2k - 2}, \delta^{2k - 2} \}}.
\end{equation}
\end{lem}

\begin{proof}
By combining the proofs of Lemma  \ref{lem: temporal regularity} and Lemma \ref{lem: spatial regularity}, we can obtain the result.
\end{proof}

\begin{rem} \label{rem: more oscillatory spatial derivatives}
We can see from Lemmas \ref{lem: temporal regularity}, \ref{lem: spatial regularity} and \ref{lem: temporal spatial regularity} that with the emergence of the multiscale potential, i.e. $0 < \delta \ll 1$, the spatial derivatives of the solution become more oscillatory than the time derivatives.
\end{rem}

\section{OC multiscale finite element basis functions for the Schr\"{o}dinger operator}
\label{sec: OC basis}

In this section, the constructions of global and localized multiscale finite element spaces will be introduced. The stationary problem \eqref{eq: steady Schrodinger equation} will be considered. And the projection errors of the solution $u$ of \eqref{eq: steady Schrodinger equation} in both the global and localized multiscale finite element spaces will be deduced. Throughout this paper, a resolution assumption is made.
\begin{asp} \label{asp: resolution assumption}
The mesh size $H$ satisfies $H / \varepsilon \le \left( 2 C_{I_H} \sqrt{C_{\text{ol}} V_{\max}(1 + C_{I_H} C_{\text{ol}} \gamma)} \right)^{-1}$.
\end{asp}

Under Assumption \ref{asp: resolution assumption}, we have a property for the kernel $W$.

\begin{lem}
Under Assumption \ref{asp: resolution assumption}, for any $v \in W$, we have
\begin{equation}
\varepsilon || \nabla v || \le || v ||_e \le C \varepsilon || \nabla v ||.
\end{equation}
\end{lem}

\begin{proof}
For any $v \in W$, it is easy to see that $\varepsilon || \nabla v || \le || v ||_e$. On the other hand, by Lemma \ref{lem: property of clement-type interpolation},
\begin{equation}
|| v ||_e^2 \le \frac{\varepsilon^2}{2} || \nabla v ||^2 + V_{\max}(v, v) \le C \left( 1 + \frac{H^2}{\varepsilon^2} \right) \varepsilon^2 || \nabla v ||^2 \le C \varepsilon^2 || \nabla v ||^2,
\end{equation}
which completes the proof.
\end{proof}

\subsection{Global multiscale finite element basis functions}

For $j = 1, \ldots, N_H$, the operator-compressed multiscale basis function $\psi_j$ is constructed as the solution of the constrained optimization problem
\begin{equation} \label{eq: global optimization problem}
\begin{aligned}
\min \limits_{\psi \in H_P^1(\Omega)} & \quad a(\psi, \psi), \\
\text{s.t.} & \quad (\psi, \phi_k) = \delta_{jk}, \quad k = 1, \ldots, N_H.
\end{aligned}
\end{equation}

Define $\Psi_H = \text{span}\{ \psi_j, j = 1, \ldots, N_H \}$ as the global multiscale finite element space. We have the following lemma.
\begin{lem} \label{lem: global quasi-orthogonality}
$H_P^1(\Omega) = \Psi_H \oplus W$ and for any $v_H \in \Psi_H$ and $w \in W$,
\begin{equation}
a(v_H, w) = 0.
\end{equation}
\end{lem}

\begin{proof}
For any nontrivial $w \in W$ and $\eta \in \mathbb{R}$, $\psi_j + \eta w$ satisfies the constraint in the optimization problem \eqref{eq: global optimization problem}, $j = 1, \ldots, N_H$. Then
\begin{equation}
g(\eta) = a(\psi_j + \eta w, \psi_j + \eta w) = \eta^2 a(w, w) + 2 \eta \text{Re } a(\psi_j, w) + a(\psi_j, \psi_j).
\end{equation}
Since $g(\eta)$ achieves the minimum at $\eta = 0$, then $g'(\eta)|_{\eta = 0} = 0$. Hence $\text{Re } a(\psi_j, w) = 0$. Set $\tilde{\eta} = i \eta$ and $\tilde{g}(\eta) = g(\tilde{\eta})$. A similar argument for $\tilde{g}(\eta)$ yields that $\text{Im } a(\psi_j, w) = 0$. Therefore, we have $a(\psi_j, w) = 0$, $j = 1, \ldots, N_H$, i.e.,
\begin{align}
a(v_H, w) = 0, \quad \forall v_H \in \Psi_H, w \in W.
\end{align}
For any $v \in H^1_P(\Omega)$, define $v^* = \sum_{k = 1}^{N_H} (v, \phi_k) \psi_k$. Then $v^* \in \Psi_H$ and
\begin{align}
(v - v^*, \phi_j) = 0, \quad j = 1, \ldots, N_H.
\end{align}
Then $v - v^* \in W$ and hence we have the decomposition $H_P^1(\Omega) = \Psi_H \oplus W$.
\end{proof}

To solve the stationary problem \eqref{eq: steady Schrodinger equation} in $\Psi_H$, the Galerkin method seeks $u_H \in \Psi_H$ such that
\begin{equation} \label{eq: Galerkin approximation in global multiscale space}
a(u_H, v_H) = (f, v_H), \quad \forall v_H \in \Psi_H.
\end{equation}
Then Lemma \ref{lem: global quasi-orthogonality} indicates the following lemma.
\begin{lem}
Assume that $u$ is the solution of \eqref{eq: steady Schrodinger equation} and $u_H$ is the solution of the Galerkin approximation \eqref{eq: Galerkin approximation in global multiscale space} in $\Psi_H$. Then $u - u_H \in W$.
\end{lem}

\begin{proof}
By Lemma \ref{lem: global quasi-orthogonality}, $u - u_H \in W$ since $a(u - u_H, w_H) = 0, \forall w_H \in \Psi_H$.
\end{proof}

We are now in the position to prove the error estimates for $u_H$.
\begin{thm} \label{thm: projection error in global space}
Let $u$ be the solution of \eqref{eq: steady Schrodinger equation} and $u_H$ be the solution of \eqref{eq: Galerkin approximation in global multiscale space}. If $f \in L^2(\Omega)$, then
\begin{equation} \label{eq: energy error in global space}
|| u - u_H ||_e \le C \frac{H}{\varepsilon} || f ||,
\end{equation}
\begin{equation} \label{eq: L2 error in global space}
|| u - u_H || \le C \frac{H^2}{\varepsilon^2} || f ||.
\end{equation}
Moreover, if $f \in H^1(\Omega)$, then
\begin{equation} \label{eq: super convergence of energy error in global space}
|| u - u_H ||_e \le C \frac{H^2}{\varepsilon} || \nabla f ||,
\end{equation}
\begin{equation} \label{eq: super convergence of L2 error in global space}
|| u - u_H || \le C \frac{H^3}{\varepsilon^2} || \nabla f ||.
\end{equation}
\end{thm}

\begin{proof}
We first consider the case where $f \in L^2(\Omega)$. For the error in the energy norm, since $u - u_H \in W$, then by Lemma \ref{lem: property of clement-type interpolation} we obtain
\begin{align}
|| u - u_H ||_e^2 & = a(u - u_H, u - u_H)  = a(u, u - u_H) \nonumber\\
& = (f, u - u_H) \le C H || f || || \nabla (u - u_H) || \le C \frac{H}{\varepsilon} || f || ||u - u_H||_e.
\end{align}
For the $L^2$ error, the Aubin-Nitsche technique is applied. Let $w \in H^1_P(\Omega)$ be the solution of
\begin{align}
a(w, v) = (u - u_H, v), \forall v \in H^1_P(\Omega)
\end{align}
and $w_H \in \Psi_H$ be the Galerkin approximation of $w$ in $\Psi_H$ satisfying
\begin{align}
a(w_H, v_H) = (u - u_H, v_H), \forall v_H \in \Psi_H.
\end{align}
Then, we easily obtain
\begin{align}
|| u - u_H ||^2 & = a(w, u - u_H)  = a(w - w_H, u - u_H) \nonumber\\
& \le || w - w_H ||_e || u - u_H ||_e  \le C \frac{H^2}{\varepsilon^2} || f || || u - u_H ||,
\end{align}
where in the last equality we have used the estimate \eqref{eq: energy error in global space}. Moreover, if $f \in H^1(\Omega)$, we have by Lemma \ref{lem: property of clement-type interpolation} that
\begin{align}
|| u - u_H ||_e^2 = (f, u - u_H) \le C H^2 || \nabla f || || \nabla (u - u_H) ||.
\end{align}
By repeating the above procedure, we can obtain the results \eqref{eq: super convergence of energy error in global space} and \eqref{eq: super convergence of L2 error in global space}.
\end{proof}

\subsection{Localized multiscale finite element basis functions}

One advantage of using these operator-compressed multiscale basis functions is that these basis functions have the property of exponential decay, which motivates us to use the localized basis functions constructed by a modified constrained energy minimization in practical computations. In this subsection, we will prove the exponential decay of the operator-compressed multiscale basis functions. Then, we will introduce the construction of localized basis functions and prove the projection error in the multiscale finite element space spanned by these localized basis functions.

\subsubsection{Exponential decay of basis functions}

Let $S_j = \text{supp}(\phi_j)$, where $\phi_j$ is the finite element nodal basis. We have the following theorem indicating the exponential decay of basis functions $\psi_j$.
\begin{thm} \label{thm: global exponential decay}
Under Assumption \ref{asp: resolution assumption}, there exists $0 < \beta < 1$ independent of $\varepsilon$, $\delta$, $m$, and $H$ such that for all $j = 1, \ldots, N_H$ and $m \in \mathbb{N}$,
\begin{equation}
|| \nabla \psi_j ||_{\Omega \backslash N^m(S_j)} \le \beta^m || \nabla \psi_j ||.
\end{equation}
\end{thm}

\begin{proof}
The proof is based on an iterative Caccioppoli-type argument \cite{hou2017sparse,Owhadi:2017,peterseim2017eliminating}. In this proof, we fix the index $j$ and omit $j$ for $\psi_j$ and $S_j$ for brevity of notations. Assume $m \ge 7$. Define the cutoff function
\begin{align}
\eta = \frac{\text{dist}(\bx, N^{m - 4}(S))}{\text{dist}(\bx, N^{m - 4}(S)) + \text{dist}(\bx, \Omega \backslash N^{m - 3}(S))}.
\end{align}
Then $\eta = 0$ in $N^{m - 4}(S)$, $\eta = 1$ in $\Omega \backslash N^{m - 3}(S)$ and $0 \le \eta \le 1$ in $N^{m - 3}(S) \backslash N^{m - 4}(S)$. Moreover, $H || \nabla \eta ||_{\infty} \le \gamma$ and $\mathcal{R} := \text{supp}(\nabla \eta) = N^{m - 3}(S) \backslash N^{m - 4}(S)$. Then, we obtain that
\begin{align}
|| \nabla \psi ||_{\Omega \backslash N^m(S)}^2 & \le (\nabla \psi, \eta \nabla \psi)  = (\nabla \psi, \nabla(\eta \psi)) - (\nabla \psi, \psi \nabla \eta) \nonumber\\
& \le | (\nabla \psi, \nabla( \eta \psi - I^{-1,\text{loc}}_H(I_H(\eta \psi)) )) | + | (\nabla \psi, \nabla I^{-1,\text{loc}}_H(I_H(\eta \psi))) | + | (\nabla \psi, \psi \nabla \eta) | \nonumber\\
& = M_1 + M_2 + M_3,
\end{align}
where $M_1  = | (\nabla \psi, \nabla( \eta \psi - I^{-1,\text{loc}}_H(I_H(\eta \psi)) )) |$,
$M_2  = | (\nabla \psi, \nabla I^{-1,\text{loc}}_H(I_H(\eta \psi))) |$, and $M_3  = | (\nabla \psi, \psi \nabla \eta) |$.  For $M_1$, note that $w = \eta \psi - I^{-1,\text{loc}}_H(I_H(\eta \psi)) \in W$, which implies $a(\psi, w) = 0$ and that $\text{supp}(w) \subset \Omega \backslash N^{m - 6}(S)$, $\text{supp}( I_H (\eta \psi) ) = N^{m - 2}(S) \backslash N^{m - 5}(S)$, $\text{supp}( \eta \psi - I_H(\eta \psi) ) \subset \Omega \backslash N^{m - 5}(S)$, and $\text{supp}( I_H(\eta \psi) - I^{-1,\text{loc}}_H(I_H(\eta \psi)) ) \subset N^{m - 1}(S) \backslash N^{m - 6}(S)$. Hence
\begin{align}
M_1 & \le \frac{2 V_{\max}}{\varepsilon^2} | (\psi, w) | = \frac{2 V_{\max}}{\varepsilon^2} | (\psi - I_H \psi, \eta \psi - I_H(\eta \psi) + I_H(\eta \psi) - I^{-1,\text{loc}}_H(I_H(\eta \psi)) ) | \nonumber\\
& \le \frac{2 V_{\max}}{\varepsilon^2} \left( | ( \psi - I_H \psi, \eta \psi - I_H(\eta \psi) ) | + | (\psi - I_H \psi, I_H(\eta \psi) - I^{-1,\text{loc}}_H(I_H(\eta \psi)) ) | \right) \nonumber\\
& \le 2 V_{\max} C_{I_H}^2 C_{\text{ol}} \frac{H^2}{\varepsilon^2} || \nabla \psi ||_{\Omega \backslash N^{m - 6}(S)} || \nabla(\eta \psi) ||_{\Omega \backslash N^{m - 6}(S)} \nonumber\\
& + 2 V_{\max} C_{I_H}^3 C_{I_H}^{\prime} C_{\text{ol}} \frac{H^2}{\varepsilon^2} || \nabla \psi ||_{N^{m}(S) \backslash N^{m - 7}(S)} || \nabla( \eta \psi ) ||_{N^{m}(S) \backslash N^{m - 7}(S)}.
\end{align}
Also note that $\mathcal{R} \cap \text{supp}(I_H \psi) = \emptyset$ and hence
\begin{align}
|| \psi \nabla \eta ||_{\mathcal{R}} = || (\psi - I_H \psi) \nabla \eta ||_{\mathcal{R}} \le C_{I_H} C_{\text{ol}} H || \nabla \eta ||_{\infty} || \nabla \psi ||_{N(\mathcal{R})} \le C_{I_H} C_{\text{ol}} \gamma || \nabla \psi ||_{N(\mathcal{R})}.
\end{align}
Thus, under Assumption \ref{asp: resolution assumption}, we arrive at
\begin{align}
M_1 \le \frac{1}{2} || \nabla \psi ||_{\Omega \backslash N^m(S)}^2 + C || \nabla \psi ||_{N^m(S) \backslash N^{m - 7}(S)}^2.
\end{align}
Using a similar argument, we have
\begin{align}
M_2 & \le C || \nabla \psi ||_{N^{m - 1}(S) \backslash N^{m - 6}(S)} || \nabla(\eta \psi) ||_{N^{m - 1}(S) \backslash N^{m - 6}(S)} \le C || \nabla \psi ||_{N^{m}(S) \backslash N^{m - 7}(S)}^2, \\
M_3 & \le C || \nabla \psi ||_{N^{m}(S) \backslash N^{m - 7}(S)}^2.
\end{align}
Combining the above estimates, we get that
\begin{align}
\frac{1}{2} || \nabla \psi ||_{\Omega \backslash N^m(S)}^2 \le C_1 || \nabla \psi ||_{N^m(S) \backslash N^{m - 7}(S)}^2,
\end{align}
where $C_1$ is independent of $\varepsilon, \delta, m$ and $H$. This leads to
\begin{align}
|| \nabla \psi ||_{\Omega \backslash N^m(S)}^2 \le \frac{C_1}{C_1 + \frac{1}{2}} || \nabla \psi ||_{\Omega \backslash N^{m - 7}(S)}^2,
\end{align}
which implies
\begin{align}
|| \nabla \psi ||_{\Omega \backslash N^m(S)}^2 \le \left( \frac{C_1}{C_1 + \frac{1}{2}} \right)^{\lfloor \frac{m}{7} \rfloor} || \nabla \psi ||^2.
\end{align}
Therefore, we prove that the basis functions $\psi_j$'s have the exponential decay property.
\end{proof}

\subsubsection{Localized basis functions}

Motivated by the exponential decay of the multiscale basis functions, we can construct the localized basis function $\psi_j^{\text{loc}, m}$ by solving the modified constrained optimization problem
\begin{equation} \label{eq: modified optimization problem}
\begin{aligned}
\min \limits_{\psi \in H_P^1(\Omega)} & \quad a(\psi, \psi), \\
\text{s.t.} & \quad (\psi, \phi_k) = \delta_{jk}, \quad k = 1, \ldots, N_H, \\
& \quad \psi(\bx) = 0, \quad \text{for } \bx \in \Omega \backslash N^m(S_j)
\end{aligned}
\end{equation}
for $j = 1, \ldots, N_H$ and $m \in \mathbb{N}$. Let $W(N^m(S_j)) = \{ w \in W | w = 0 \text{ in } \Omega \backslash N^m(S_j) \}$. Following the proof of Lemma \ref{lem: global quasi-orthogonality}, we can prove the following lemma.
\begin{lem}
It holds true that for $j = 1, \ldots, N_H$
\begin{equation}
a(\psi_j^{\text{loc}, m}, w) = 0, \quad \forall w \in W(N^m(S_j)).
\end{equation}
\end{lem}

Define $\Psi_{H, m} = \text{span}\{\psi_j^{\text{loc}, m}, j = 1, \ldots, N_H\}$ as the localized multiscale finite element space. Before we study the projection error in $\Psi_{H,m}$, a lemma on the bound of $|| \nabla \psi_j ||$ is needed.
\begin{lem} \label{lem: bound of global basis}
Under Assumption \ref{asp: resolution assumption}, it holds true that for $j = 1, \ldots, N_H$
\begin{equation}
|| \nabla \psi_j || \le C H^{-\frac{3}{2}d}.
\end{equation}
\end{lem}

\begin{proof}
Define the operator $P$ as for any $v \in H^1_P(\Omega)$, $P v \in W$ and
\begin{align}
a(Pv, w) = a(v, w), \quad \forall w \in W.
\end{align}
By Lax-Milgram theorem, $P$ is well defined and $|| Pv ||_e \le || v ||_e$. Let $\hat{\psi}_j = P \phi_j - \phi_j$. Then $\hat{\psi}_j \in \Psi_H$ since $P \hat{\psi}_j = 0$, $j = 1, \ldots, N_H$. We know that $\{ \hat{\psi}_j \}_{j = 1}^{N_H}$ spans $\Psi_H$ since $\hat{\psi}_j$'s are linearly independent. Therefore
\begin{align}
\psi_j = \sum_{k = 1}^{N_H} \alpha_k^{(j)} (P \phi_k - \phi_k).
\end{align}
Note that $(\psi_j, \phi_{\ell}) = \delta_{j,\ell}$. Then $\sum_{k = 1}^{N_H} \alpha_k^{(j)}(\phi_k, \phi_{\ell}) = - \delta_{j, \ell}$. So if we let $\alpha^{(j)} = (\alpha_1^{(j)}, \ldots, \alpha_{N_H}^{(j)})$, then $M \alpha^{(j)} = - e_j$, where $e_j$ is a column vector with the $j$-th entry as $1$ and other entries as 0 and $M$ is the mass matrix with entries $M_{j,k} = (\phi_j, \phi_k)$. From the results in \cite{hackbusch2015hierarchical,hackbusch2017elliptic}, we know that $| (M^{-1})_{j,k} | \le C H^{-d}$.
Therefore, under Assumption \ref{asp: resolution assumption}, we have
\begin{align}
|| \nabla \psi_j || \le \sum_{k = 1}^{N_H} \alpha_k^{(j)} || \nabla(P \phi_k - \phi_k) || \le C N_H H^{-\frac{d}{2}} \le C H^{-\frac{3}{2} d}.
\end{align}
\end{proof}
We also need two lemmas on the difference between $\psi_j$ and $\psi_j^{\text{loc}, m}$.
\begin{lem}
Under Assumption \ref{asp: resolution assumption}, for $j = 1, \ldots, N_H$, we have
\begin{equation}
|| \nabla(\psi_j - \psi_j^{\text{loc}, m}) || \le C H^{-\frac{3}{2}d} \beta^m.
\end{equation}
\end{lem}

\begin{proof}
Let $m \ge 6$ and $\tilde{\psi}_j = \psi_j - I^{-1,\text{loc}}_H( I_H \psi_j )$ and $\tilde{\psi}_j^{\text{loc}, m} = \psi_j^{\text{loc}, m} - I^{-1,\text{loc}}_H ( I_H \psi_j^{\text{loc}, m} )$. Then
$\psi_j - \psi_j^{\text{loc}, m} = \tilde{\psi}_j - \tilde{\psi}_j^{\text{loc}, m}$ since $I_H \psi_j = I_H \psi_j^{\text{loc}, m} = \phi_j / (1, \phi_j)$. In addition, $\tilde{\psi}_j \in W$ and $\tilde{\psi}_j^{\text{loc}, m} \in W(N^m(S_j))$. Then $\forall w \in W(N^m(S_j))$,
\begin{align}
\varepsilon^2 || \nabla( \tilde{\psi}_j - \tilde{\psi}_j^{\text{loc}, m} ) ||^2 & \le a(\tilde{\psi}_j - \tilde{\psi}_j^{\text{loc}, m}, \tilde{\psi}_j - \tilde{\psi}_j^{\text{loc}, m}) = a(\tilde{\psi}_j - \tilde{\psi}_j^{\text{loc}, m}, \tilde{\psi}_j - w) \nonumber\\
& \le || \tilde{\psi}_j - \tilde{\psi}_j^{\text{loc}, m} ||_e || \tilde{\psi}_j - w ||_e \le C \varepsilon^2 || \nabla( \tilde{\psi}_j - \tilde{\psi}_j^{\text{loc}, m} ) || || \nabla( \tilde{\psi}_j - w  ) ||,
\end{align}
where we have used the fact that $a(\psi_j, w - \tilde{\psi}_j^{\text{loc}, m}) = a(\psi_j^{\text{loc}, m}, w- \tilde{\psi}_j^{\text{loc}, m}) = 0$ and hence $a(\tilde{\psi}_j - \tilde{\psi}_j^{\text{loc}, m}, w - \tilde{\psi}_j^{\text{loc}, m}) = a(\psi_j - \psi_j^{\text{loc}, m}, w - \tilde{\psi}_j^{\text{loc}, m}) = 0$.

Define the cutoff function
\begin{align}
\eta = \frac{\text{dist}(\bx, \Omega \backslash N^{m - 2}(S_j))}{\text{dist}(\bx, N^{m - 3}(S_j)) + \text{dist}(\bx, \Omega \backslash N^{m - 2}(S_j))}.
\end{align}
Then $\eta = 1$ in $N^{m - 3}(S_j)$, $\eta = 0$ in $\Omega \backslash N^{m - 2}(S_j)$ and $0 \le \eta \le 1$ in $N^{m - 2}(S_j) \backslash N^{m - 3}(S_j)$. Furthermore, $H || \nabla \eta ||_{\infty} \le \gamma$ and $\text{supp}(\nabla \eta) = N^{m - 2}(S_j) \backslash N^{m - 3}(S_j)$.

Take $w = \eta \tilde{\psi}_j - I^{-1,\text{loc}}_H(I_H( \eta \tilde{\psi}_j )) \in W(N^m(S_j))$. Then, we can prove that
\begin{align}
|| \nabla( \psi_j - \psi_j^{\text{loc}, m} ) || & = || \nabla( \tilde{\psi}_j - \tilde{\psi}_j^{\text{loc}, m} ) ||^2  \le C || \nabla( \tilde{\psi}_j - w ) ||^2 \le C \varepsilon^{-2} || \tilde{\psi}_j - w ||_e^2 \nonumber\\
& \le C \Big( || \nabla( (1 - \eta) \tilde{\psi}_j ) ||^2 + \frac{1}{\varepsilon^{2}} || (1 - \eta) \tilde{\psi}_j ||^2 \nonumber\\
& + || \nabla( \eta \tilde{\psi}_j ) ||^2_{N^m(S_j) \backslash N^{m - 5}(S_j)} + \frac{1}{\varepsilon^{2}} || \eta \tilde{\psi}_j ||^2_{N^m(S_j) \backslash N^{m - 5}(S_j)} \Big) \nonumber\\
& \le C \left( 1 + \frac{H^2}{\varepsilon^{2}} + H^2 || \nabla \eta ||_{\infty}^2 \right) || \nabla \psi_j ||^2_{\Omega \backslash N^{m - 5}(S_j)} + C \frac{H^2}{\varepsilon^{2}} || \nabla \psi_j ||^2_{\Omega \backslash N^{m - 6}(S_j)} \nonumber\\
& \le C || \nabla \psi_j ||^2_{\Omega \backslash N^{m - 6}(S_j)},
\end{align}
which completes the proof with Theorem \ref{thm: global exponential decay} and Lemma \ref{lem: bound of global basis}.
\end{proof}
\begin{lem} \label{lem: comparison of global and local projection}
Let $v \in H^1_P(\Omega)$ and $v_1 = \sum_{k = 1}^{N_H} (v, \phi_k) \psi_k, v_2 = \sum_{k = 1}^{N_H} (v, \phi_k) \psi_k^{\text{loc}, m}$. Then under Assumption \ref{asp: resolution assumption}, we have
\begin{equation}
|| v_1 - v_2 ||_e \le C \varepsilon H^{-\frac{3}{2}d} \beta^m || v ||_e.
\end{equation}
\end{lem}
\begin{proof}
Note that $v_1 - v_2 \in W$. Then
\begin{align}
|| v_1 - v_2 ||_e & \le C \varepsilon || \nabla( v_1 - v_2 ) || \le C \varepsilon \sum_{k = 1}^{N_H} | (v, \phi_k) | || \nabla( \psi_k - \psi_k^{\text{loc}, m} ) || \nonumber\\
& \le C \varepsilon H^{-\frac{3}{2}d} \beta^m (|v|, 1) \le C \varepsilon H^{-\frac{3}{2}d} \beta^m || v ||_e.
\end{align}
\end{proof}

Similarly, the Galerkin approximation of \eqref{eq: steady Schrodinger equation} in $\Psi_{H,m}$ is to seek $u_{H,m} \in \Psi_{H,m}$ such that
\begin{equation} \label{eq: Galerkin approximation in localized multiscale space}
a(u_{H,m}, v_{H,m}) = (f, v_{H,m}), \quad \forall v_{H,m} \in \Psi_{H,m}.
\end{equation}

In order to obtain the projection error estimate for the localized multiscale finite element space $\Psi_{H, m}$, we need the following assumption on the oversampling size $m$.
\begin{asp} \label{asp: oversampling assumption}
The oversamping size $m$ satisfies
\begin{equation} \label{eq: oversampling assumption}
m \ge \frac{ C_d \log (1 / H) + \log ( \varepsilon^2 C_{\text{st}}(\varepsilon, V) ) } { | \log(\beta) | },
\end{equation}
where $C_d = 3d / 2 + 2$.
\end{asp}

Then we have an error estimate for $u - u_{H,m}$ as follows.
\begin{thm} \label{thm: projection error in localized space}
Assume that Assumptions \ref{asp: resolution assumption} and \ref{asp: oversampling assumption} hold and let $u$ the solution of \eqref{eq: steady Schrodinger equation} and $u_{H,m}$ be the solution of \eqref{eq: Galerkin approximation in localized multiscale space}. If $f \in L^2(\Omega)$, then
\begin{equation} \label{eq: energy error in local space}
|| u - u_{H,m} ||_e \le C \frac{H}{\varepsilon} || f ||,
\end{equation}
\begin{equation} \label{eq: L2 error in local space}
|| u - u_{H,m} || \le C \frac{H^2}{\varepsilon^2} || f ||.
\end{equation}
Moreover, If $f \in H^1(\Omega)$, then
\begin{equation} \label{eq: super convergence of energy error in local space}
|| u - u_{H,m} ||_e \le C \frac{H^2}{\varepsilon} || f ||_{H^1},
\end{equation}
\begin{equation} \label{eq: super convergence of L2 error in local space}
|| u - u_{H,m} || \le C \frac{H^3}{\varepsilon^2} || f ||_{H^1}.
\end{equation}
\end{thm}

\begin{proof}
We first consider the case where $f \in L^2(\Omega)$. Let $\tilde{u}_{H,m} = \sum_{k = 1}^{N_H} (u, \phi_k) \psi_k^{\text{loc}, m}$. Then it is easy to verify that
\begin{align}
|| u - u_{H,m} ||_e \le || u - \tilde{u}_{H, m} ||_e.
\end{align}
Set $u_H = \sum_{k = 1}^{N_H} (u, \phi_k) \psi_k$. Then since $u - \tilde{u}_{H,m} \in W$, $a(u_H, u - \tilde{u}_{H,m}) = 0$ and
\begin{align}
|| u - \tilde{u}_{H, m} ||_e^2 & = a(u - \tilde{u}_{H, m}, u - \tilde{u}_{H, m}) = a(u, u - \tilde{u}_{H, m}) + a(u_H - \tilde{u}_{H, m}, u - \tilde{u}_{H, m}) \nonumber\\
& = (f, u - \tilde{u}_{H, m}) + a(u_H - \tilde{u}_{H, m}, u - \tilde{u}_{H, m}) \nonumber\\
& \le C H || f || || \nabla( u - \tilde{u}_{H, m} ) || + || u_H - \tilde{u}_{H, m} ||_e || u - \tilde{u}_{H, m} ||_e \nonumber\\
& \le C \frac{H}{\varepsilon} || f || || u - \tilde{u}_{H, m} ||_e + C \varepsilon H^{-\frac{3}{2}d} \beta^m || u ||_e || u - \tilde{u}_{H, m} ||_e \nonumber\\
& \le C \frac{H}{\varepsilon} || f || || u - \tilde{u}_{H, m} ||_e + C \varepsilon H^{-\frac{3}{2}d} C_{\text{st}}(\varepsilon, V) \beta^m || f || || u - \tilde{u}_{H, m} ||_e.
\end{align}
Hence if $m$ satisfies Assumption \ref{asp: oversampling assumption}, then
\begin{align}
|| u - u_{H,m} ||_e \le C \frac{H}{\varepsilon} || f ||.
\end{align}
A similar Aubin-Nitsche technique to the proof of Theorem \ref{thm: projection error in global space} can be applied to obtain \eqref{eq: L2 error in local space}. Moreover, if $f \in H^1(\Omega)$, we have by Lemma \ref{lem: property of clement-type interpolation} that
\begin{align}
|| u - \tilde{u}_{H, m} ||_e^2 & = (f, u - \tilde{u}_{H, m}) + a(u_H - \tilde{u}_{H, m}, u - \tilde{u}_{H, m}) \nonumber\\
& \le C \frac{H^2}{\varepsilon} || \nabla f || || u - \tilde{u}_{H, m} ||_e + C \varepsilon H^{-\frac{3}{2}d} C_{\text{st}}(\varepsilon, V) \beta^m || f || || u - \tilde{u}_{H, m} ||_e.
\end{align}
Hence if $m$ satisfies Assumption \ref{asp: oversampling assumption}, we can obtain
\begin{align}
|| u - u_{H,m} ||_e \le C \frac{H^2}{\varepsilon} || f ||_{H^1}.
\end{align}
Analogously, we obtain \eqref{eq: super convergence of L2 error in local space} using the Aubin-Nitsche technique.
\end{proof}

\begin{rem}
To obtain \eqref{eq: energy error in local space} and \eqref{eq: L2 error in local space}, it is sufficient to assume that $m$ satisfies \eqref{eq: oversampling assumption} with $C_d = 3d / 2 + 1$. We impose a stronger assumption on $m$ in order to avoid lengthy illustration of Theorem \ref{thm: projection error in localized space}.
\end{rem}

\section{Convergence of the OC MsFEM for Schr\"{o}dinger equations with  multiscale potentials}
\label{sec: convergence analysis}

In this section, we will study the error estimate of the OC MsFEM for the Schr\"{o}dinger equation \eqref{eq: time-dependent Schrodinger equation}, where the Crank-Nicolson scheme is used for temporal discretization. Throughout this section, we will not distinguish between the global multiscale finite element space $\Psi_H$ and the localized multiscale finite element space $\Psi_{H, m}$ with $H$ satisfying Assumption \ref{asp: resolution assumption} and $m$ satisfying Assumption \ref{asp: oversampling assumption}. Both of the spaces will be denoted by $\Psi_H$.

\subsection{Projection error}\label{sec: Projection error}

Let $u(t)$ be the solution of the Schr\"{o}dinger equation \eqref{eq: time-dependent Schrodinger equation} and $\hat{u}(t)$ be the projection of $u(t)$ in $\Psi_H$ such that $\forall 0 \le t \le T, \hat{u}(t) \in \Psi_H$ and
\begin{equation}
a(u(t) - \hat{u}(t), w) = 0, \quad \forall w \in \Psi_H.
\end{equation}
Then, we have the following lemmas on the projection errors.
\begin{lem} \label{lem: projection error}
If $u_t(t) \in L^2(\Omega)$ for any $t \in [0, T]$, then it holds true that for any $0 \le t \le T$
\begin{equation}
|| u(t) - \hat{u}(t) ||_e \le C \frac{H}{\varepsilon}, \quad \text{and},
\quad || u(t) - \hat{u}(t) || \le C \frac{H^2}{\varepsilon^2}.
\end{equation}
Moreover, if $u_t(t) \in H^1(\Omega)$ for any $t \in [0, T]$, then it holds true that for any $0 \le t \le T$
\begin{equation}
|| u(t) - \hat{u}(t) ||_e \le C \frac{H^2}{\varepsilon^2 \delta},\quad \text{and},
\quad
|| u(t) - \hat{u}(t) || \le C \frac{H^3}{\varepsilon^3 \delta}.
\end{equation}
\end{lem}

\begin{proof}
We first consider the case where $u_t(t) \in L^2(\Omega), \forall t \in [0, T]$. By \eqref{eq: time-dependent Schrodinger equation}, Theorem  \ref{thm: projection error in global space}, Theorem  \ref{thm: projection error in localized space} and Lemma \ref{lem: temporal regularity}, we have that for any $0 \le t \le T$,
\begin{align}
& || u(t) - \hat{u}(t) ||_e \le C \frac{H}{\varepsilon} || \Hc u(t) || \le C H || u_t(t) || \le C \frac{H}{\varepsilon}, \\
& || u(t) - \hat{u}(t) || \le C \frac{H^2}{\varepsilon^2} || \Hc u(t) || \le C \frac{H^2}{\varepsilon} || u_t(t) || \le C \frac{H^2}{\varepsilon^2}.
\end{align}
Moreover, if $u_t(t) \in H^1(\Omega)$ for any $t \in [0, T]$, by \eqref{eq: time-dependent Schrodinger equation}, Theorem \ref{thm: projection error in global space}, Theorem \ref{thm: projection error in localized space}, Lemma \ref{lem: temporal regularity}, and Lemma \ref{lem: temporal spatial regularity}, we know that for any $0 \le t \le T$,
\begin{align}
& || u(t) - \hat{u}(t) ||_e \le C \frac{H^2}{\varepsilon} || \Hc u(t) ||_{H^1} \le C H^2 || u_t(t) ||_{H^1} \le C \frac{H^2}{\varepsilon^2 \delta}, \\
& || u(t) - \hat{u}(t) || \le C \frac{H^3}{\varepsilon^2} || \Hc u(t) ||_{H^1} \le C \frac{H^3}{\varepsilon} || u_t(t) ||_{H^1} \le C \frac{H^3}{\varepsilon^3 \delta}.
\end{align}
\end{proof}

\begin{lem} \label{lem: projection error of time derivatives}
If $\partial_t^{k + 1} u(t) \in L^2(\Omega)$ for any $t \in [0, T]$ and $k = 1, 2$, then it holds true that for any $0 \le t \le T$
\begin{equation}
|| \partial_t^k u(t) - \partial_t^k \hat{u}(t) ||_e \le \frac{C H}{\varepsilon^{1 - k} \min\{ \varepsilon^{2k}, \delta^{2k} \}},
\quad || \partial_t^k u(t) - \partial_t^k \hat{u}(t) || \le \frac{C H^2}{\varepsilon^{2 - k} \min\{ \varepsilon^{2k}, \delta^{2k} \}}.
\end{equation}
Moreover, if $\partial_t^{k + 1} u(t) \in H^1(\Omega)$ for any $t \in [0, T]$ and $k = 1, 2$, then it holds true that for any $0 \le t \le T$
\begin{equation}
|| \partial_t^k u(t) - \partial_t^k \hat{u}(t) ||_e \le \frac{C H^2}{\varepsilon^{2 - k} \delta \min\{ \varepsilon^{2k}, \delta^{2k} \}}, \quad
|| \partial_t^k u(t) - \partial_t^k \hat{u}(t) || \le \frac{C H^3}{\varepsilon^{3 - k} \delta \min\{ \varepsilon^{2k}, \delta^{2k} \}}.
\end{equation}
\end{lem}
The proof of Lemma \ref{lem: projection error of time derivatives} is similar to Lemma \ref{lem: projection error}.
We can easily see that higher regularity of the solution $u$ will lead to super convergence of the projection errors in $\Psi_H$ w.r.t. $H$.

We can also study the error of the FEM in solving the Schr\"{o}dinger equation \eqref{eq: time-dependent Schrodinger equation}.  Let $\tilde{u}$ be the projection of $u$ in the standard linear finite element space $\Phi_H$ such that $\forall 0 \le t \le T$, $\tilde{u} \in \Phi_H$ and
\begin{equation}
		a(u(t) - \tilde{u}(t), w) = 0, \quad \forall w \in \Phi_H.
\end{equation}
Then
\begin{equation}
		|| u(t) - \tilde{u}(t) ||_e \le \inf \limits_{w \in \Phi_H} || u(t) - w ||_e.
\end{equation}
Let $\chi(t)$ be the interpolation of $u(t)$ in $\Phi_H$ such that $\chi = \sum_{\bx_k \in \oN_H} u(t, \bx_k) \phi_k$, where $\phi_j(\bx_k) = \delta_{j,k}$. A well known result for the errors of the interpolation \cite{brenner2007mathematical,ciarlet2002finite,thomee2007galerkin} is that
	\begin{equation}
		|| \chi(t) - u(t) || \le C H^2 || u(t) ||_{H^2}, \text{ and }, || \nabla( \chi(t) - u(t) ) || \le C H || u(t) ||_{H^2}
	\end{equation}
	and hence
	\begin{equation}
		|| u(t) - \tilde{u}(t) ||_e \le || \chi(t) - u(t) ||_e \le C \varepsilon H \sqrt{1 + \frac{H^2}{\varepsilon^2}} || u(t) ||_{H^2}.
	\end{equation}
By Lemma \ref{lem: spatial regularity} and under Assumption \ref{asp: resolution assumption}, we know that for any $0 \le t \le T$,
	\begin{equation} \label{eq: energy error of projection in FEM}
		|| u(t) - \tilde{u}(t) ||_e \le C \frac{H}{\varepsilon \delta^2}.
	\end{equation}
Using the Aubin-Nitsche technique with $H^2$ regularity of elliptic equations \cite{evans1998partial,hackbusch2017elliptic}, we obtain that for any $0 \le t \le T$,
	\begin{equation} \label{eq: L2 error of projection in FEM}
		|| u(t) - \tilde{u}(t) || \le C H^2 || u ||_{H^2} \le C \frac{H^2}{\varepsilon^2 \delta^2}.
	\end{equation}
In comparison with Lemma \ref{lem: projection error}, we find that the projection error estimate of the OC MsFEM depends on higher temporal regularity of the solution $u$, while that of FEM depends on higher spatial regularity of $u$. Lemmas \ref{lem: temporal regularity}, \ref{lem: spatial regularity} and \ref{lem: temporal spatial regularity} show that in the presence of the multiscale potential $V^{\delta}$, the spatial derivatives of $u$ become more oscillatory than the time derivatives of $u$. Therefore, using the same mesh size the OC MsFEM gives more accurate results than the FEM in solving Schr\"{o}dinger equation with multiscale potentials due to its super convergence behavior for high-regularity solutions and weaker dependence on the small parameters $\varepsilon$ and $\delta$ for low-regularity solutions.

\subsection{Crank-Nicolson OC MsFEM scheme}

We analyse the error estimate of the Crank-Nicolson OC MsFEM scheme for \eqref{eq: time-dependent Schrodinger equation}, where we use the Crank-Nicolson scheme in temporal discretization and the OC MsFEM in spatial discretization.
In the following, we introduce some notations. For some $N \in \mathbb{N}$ and $N > 0$, let $\Delta t = \frac{T}{N}$ and $t_n = n \Delta t, n = 0, 1, \ldots, N$.

We approximate $u(t_n)$ by $U^n \in \Psi_H$ such that
\begin{equation} \label{eq: Crank-Nicolson scheme}
\begin{aligned}
i \varepsilon (\bar{\partial} U^n, w) & = a\left(\frac{U^n + U^{n - 1}}{2}, w\right), \quad \forall w \in \Psi_H, n = 1, \ldots, N, \\
U^0 & = \hat{u}(0),
\end{aligned}
\end{equation}
where $\bar{\partial} U^n = \frac{U^n - U^{n - 1}}{\Delta t}$. We let $U^n - u(t_n) = \theta^n + \rho^n$, where $\theta^n = U^n - \hat{u}(t_n)$ and $\rho^n = \hat{u}(t_n) - u(t_n)$. Then,  $\theta^n$ satisfies that $\theta^0 = 0$ and
\begin{equation} \label{eq: evolution of error in Crank-Nicolson}
i \varepsilon (\bar{\partial} \theta^n, w) + i \varepsilon (z^n_1, w) + i \varepsilon (z^n_2, w) = a\left(\frac{\theta^n + \theta^{n - 1}}{2}, w\right), \quad \forall w \in \Psi_H,\quad  n = 1, \ldots, N,
\end{equation}
where $\bar{\partial} \theta^n = \frac{\theta^n - \theta^{n - 1}}{\Delta t}$, $z^n_1 = \bar{\partial} \hat{u}(t_n) - \bar{\partial} u(t_n)$ and $z^n_2 = \bar{\partial} u(t_n) - \frac{u_t(t_n) + u_t(t_{n - 1})}{2}$ with $\bar{\partial} \hat{u}(t_n) = \frac{\hat{u}(t_n) - \hat{u}(t_{n - 1})}{\Delta t}$ and $\bar{\partial} {u}(t_n) = \frac{{u}(t_n) - {u}(t_{n - 1})}{\Delta t}$.

For the $L^2$ error of $U^N$, we have the following theorem.
\begin{thm} \label{thm: L2 error in Crank-Nicolson}
Assume that $U^N$ is the solution of \eqref{eq: Crank-Nicolson scheme} and $u$ is the solution of \eqref{eq: time-dependent Schrodinger equation}. If $u_t(t), u_{tt}(t), u_{ttt}(t) \in L^2(\Omega)$ for any $t \in [0, T]$, then
\begin{equation}
|| U^N - u(T) || \le C \left( \frac{\varepsilon \Delta t^2}{\min\{ \varepsilon^4, \delta^4 \}} + \frac{H^2}{\varepsilon \min\{ \varepsilon^2, \delta^2 \}} \right).
\end{equation}
Moreover, If $u_t(t), u_{tt}(t) \in H^1(\Omega)$ for any $t \in [0, T]$, then
\begin{equation} \label{eq: super convergence of L2 error in Crank Nicolson}
|| U^N - u(T) || \le C \left( \frac{\varepsilon \Delta t^2}{\min\{ \varepsilon^4, \delta^4 \}} + \frac{H^3}{\varepsilon^2 \delta \min\{ \varepsilon^2, \delta^2 \}} \right).
\end{equation}
\end{thm}

\begin{proof}
We first consider the case where $u_t(t), u_{tt}(t), u_{ttt}(t) \in L^2(\Omega), \forall t \in [0, T]$. We have $|| U^n - u(t_n) || \le || \theta^n || + || \rho^n ||$ and $|| \rho^N || \le C \frac{H^2}{\varepsilon^2}$, $\theta^0 = 0$. Setting $w = \theta^n + \theta^{n - 1}$ in \eqref{eq: evolution of error in Crank-Nicolson} and taking the imaginary part of it, we have $\text{Re}(\bar{\partial} \theta^n, \theta^n + \theta^{n - 1}) = - \text{Re} (z^n_1 + z^n_2, \theta^n + \theta^{n - 1})$,
which implies
\begin{align}
\frac{1}{\Delta t} (|| \theta^n ||^2 - || \theta^{n - 1} ||^2) \le (|| \theta^n || + || \theta^{n - 1} ||) (|| z^n_1 || + || z^n_2 ||).
\end{align}
Therefore, we have
\begin{align}
|| \theta^N || \le || \theta^0 || + \Delta t \sum_{n = 1}^N (|| z^n_1 || + || z^n_2 ||).
\end{align}
For $|| z^n_1 ||$, $n = 1, 2, \ldots, N$,
\begin{align}
|| z^n_1 || & = || \bar{\partial} \hat{u}(t_n) - \bar{\partial} u(t_n) ||
\le \frac{1}{\Delta t} \int_{t_{n - 1}}^{t_n} || \hat{u}_t(s) - u_t(s) || ds \nonumber\\
& \le \frac{C}{\Delta t} \frac{H^2}{\varepsilon} \int_{t_{n - 1}}^{t_n} || u_{tt}(s) || ds \le \frac{C H^2}{\varepsilon \min\{ \varepsilon^2, \delta^2 \}}.
\end{align}
This gives us that $ \Delta t \sum_{n = 1}^N || z^n_1 || \le \frac{C H^2}{\varepsilon \min\{ \varepsilon^2, \delta^2 \}}$. For $|| z^n_2 ||$, $n = 1, 2, \ldots, N$, we have
\begin{align}
|| z^n_2 || & = \frac{1}{2 \Delta t} || 2 ( u(t_n) - u(t_{n - 1}) ) - \Delta t ( u_t(t_n) + u_t(t_{n - 1}) ) || \nonumber\\
& \le \frac{1}{2 \Delta t} \Big( \int_{t_{n - \frac{1}{2}}}^{t_n} (t_n - s)^2 || u_{ttt}(s) || ds + \int_{t_{n - 1}}^{t_{n - \frac{1}{2}}} (s - t_{n - 1})^2 || u_{ttt}(s) || ds \nonumber\\
& + \Delta t \int_{t_{n - \frac{1}{2}}}^{t_n} (t_n - s) || u_{ttt}(s) || ds + \Delta t \int_{t_{n - 1}}^{t_{n - \frac{1}{2}}} (s - t_{n - 1}) || u_{ttt}(s) || ds \Big) \nonumber\\
& \le C \Delta t^2 \max \limits_{0 \le t \le T} || u_{ttt}(t) || \le \frac{C \varepsilon \Delta t^2}{\min\{ \varepsilon^4, \delta^4 \}},
\end{align}
and hence $\Delta t \sum_{n = 1}^N || z^n_2 || \le \frac{C \varepsilon \Delta t^2}{\min\{ \varepsilon^4, \delta^4 \}}$.
Combining all the above inequalities, we obtain
\begin{align}
|| U^N - u(T) || \le C \left( \frac{\varepsilon \Delta t^2}{\min\{ \varepsilon^4, \delta^4 \}} + \frac{H^2}{\varepsilon \min\{ \varepsilon^2, \delta^2 \}} \right).
\end{align}
Moreover, if $u_t(t), u_{tt}(t) \in H^1(\Omega)$ for any $t \in [0, T]$, then we have $|| \rho^N || \le C \frac{H^3}{\varepsilon^3 \delta}$ and for $n = 1, 2, \ldots, N$,
\begin{align}
|| z^n_1 || \le \frac{C}{\Delta t} \frac{H^3}{\varepsilon} \int_{t_{n - 1}}^{t_n} || u_{tt}(s) ||_{H^1} ds \le \frac{C H^3}{\varepsilon^2 \delta \min\{ \varepsilon^2, \delta^2 \}}.
\end{align}
Using similar arguments, we can prove the estimate in \eqref{eq: super convergence of L2 error in Crank Nicolson}.
\end{proof}

For the error in the energy norm, we have the following theorem.
\begin{thm} \label{thm: energy error in Crank-Nicolson}
Assume that $U^N$ is the solution of \eqref{eq: Crank-Nicolson scheme} and $u$ is the solution of \eqref{eq: time-dependent Schrodinger equation}. If $u_t(t), u_{tt}(t), u_{ttt}(t), u_{tttt}(t) \in L^2(\Omega)$ for any $t \in [0, T]$, then
\begin{equation}
|| U^N - u(T) ||_e \le C \left( \frac{H}{\varepsilon} + \frac{H^2}{\min\{ \varepsilon^3, \delta^3 \}} + \frac{\varepsilon^2 \Delta t^2}{\min\{ \varepsilon^5, \delta^5 \}} \right).
\end{equation}
Moreover, if $u_t(t), u_{tt}(t), u_{ttt}(t) \in H^1(\Omega)$ for any $t \in [0, T]$, then
\begin{equation} \label{eq: super convergence of energy error in Crank Nicolson}
|| U^N - u(T) ||_e \le C \left( \frac{H^2}{\varepsilon^2 \delta} + \frac{H^3}{\varepsilon \delta \min\{ \varepsilon^3, \delta^3 \}} + \frac{\varepsilon^2 \Delta t^2}{\min\{ \varepsilon^5, \delta^5 \}} \right).
\end{equation}
\end{thm}

\begin{proof}
We first consider the case where $u_t(t), u_{tt}(t), u_{ttt}(t), u_{tttt}(t) \in L^2(\Omega), \forall t \in [0, T]$. We have $|| U^n - u(t_n) ||_e \le || \theta^n ||_e + || \rho^n ||_e$, where $|| \rho^N ||_e \le C \frac{H}{\varepsilon}$. Setting $w = \theta^n - \theta^{n - 1}$ in \eqref{eq: evolution of error in Crank-Nicolson} and taking the real part of it, we have
\begin{align}
|| \theta^n ||_e^2 \le || \theta^{n - 1} ||_e^2 + 2 \varepsilon |(z^n_1 + z^n_2, \theta^n - \theta^{n - 1})| \le || \theta^{n - 1} ||_e^2 + 2 \varepsilon || \theta^n - \theta^{n - 1} || ( || z^n_1 || + || z^n_2 || ).
\end{align}
For $ || \theta^n - \theta^{n - 1} || $, we can derive by \eqref{eq: evolution of error in Crank-Nicolson} that
\begin{align}
i \varepsilon (\bar{\partial} \theta^n - \bar{\partial} \theta^{n - 1}, w) + i \varepsilon (z^n_1 - z^{n - 1}_1, w) + i \varepsilon (z^n_2 - z^{n - 1}_2, w) = a \left( \frac{\theta^n - \theta^{n - 2}}{2}, w \right), \quad \forall w \in \Psi_H.
\end{align}
Setting $w = \bar{\partial} \theta^n + \bar{\partial} \theta^{n - 1} = \frac{\theta^n - \theta^{n - 2}}{\Delta t}$ in the last equality and taking the imaginary part of it, we have
\begin{align}
|| \bar{\partial} \theta^n ||^2 - || \bar{\partial} \theta^{n - 1} ||^2 \le (|| \bar{\partial} \theta^n || + || \bar{\partial} \theta^{n - 1} ||)(|| z^n_1 - z^{n - 1}_1 || + || z^n_2 - z^{n - 1}_2 ||)
\end{align}
and hence
\begin{align}
||\theta^n - \theta^{n - 1}|| \le || \theta^1 - \theta^0 || + \Delta t \sum_{j = 2}^n (|| z^j_1 - z^{j - 1}_1 || + || z^j_2 - z^{j - 1}_2 ||).
\end{align}
For $|| \theta^1 - \theta^0 ||$, we have by the proof of Theorem \ref{thm: L2 error in Crank-Nicolson} that
\begin{align}
|| \theta^1 - \theta^0 || = || \theta^1 || \le \Delta t (|| z^1_1 || + || z^1_2 ||) \le C \left( \frac{\varepsilon \Delta t^3}{\min\{ \varepsilon^4, \delta^4 \}} + \frac{\Delta t H^2}{\varepsilon \min\{ \varepsilon^2, \delta^2 \}} \right).
\end{align}
For $\Delta t \sum_{j = 2}^n || z^j_1 - z^{j - 1}_1 ||$, we have
\begin{align}
|| z^j_1 - z^{j - 1}_1 || & = \frac{1}{\Delta t} || \rho(t_{j}) - 2 \rho(t_{j - 1}) + \rho(t_{j - 2}) || \nonumber\\
& \le \frac{1}{\Delta t} \left( \int_{t_{j - 1}}^{t_j} (t_j - s) || \rho_{tt}(s) || ds + \int_{t_{j - 2}}^{t_{j - 1}} (s - t_{j - 2}) || \rho_{tt}(s) || ds \right) \nonumber\\
& \le C \frac{\Delta t H^2}{\varepsilon} \max \limits_{0 \le t \le T} || u_{ttt}(t) ||  \le \frac{C \Delta t H^2}{ \min\{ \varepsilon^4, \delta^4 \} }
\end{align}
and hence $\Delta t \sum_{j = 2}^n || z^j_1 - z^{j - 1}_1 || \le \frac{C \Delta t H^2}{ \min\{ \varepsilon^4, \delta^4 \} }$. For the term $\Delta t \sum_{j = 2}^n || z^j_2 - z^{j - 1}_2 ||$, we have
\begin{align}
|| z^j_2 - z^{j - 1}_2 || & = \frac{1}{2 \Delta t} || 2(u(t_j) - 2 u(t_{j - 1}) + u(t_{j - 2})) - \Delta t (u_t(t_j) - u_t(t_{j - 2})) || \nonumber\\
& \le \frac{1}{12 \Delta t} \Big( 2 \int_{t_{j - 1}}^{t_j} (t_j - s)^3 || u_{tttt}(s) || ds +  2 \int_{t_{j - 2}}^{t_{j - 1}} (s - t_{j - 2})^3 || u_{tttt}(s) || ds \nonumber\\
& + 3 \Delta t \int_{t_{j - 1}}^{t_j} (t_j - s)^2 || u_{tttt}(s) || ds + 3 \Delta t \int_{t_{j - 2}}^{t_{j - 1}} (s - t_{j - 2})^2 || u_{tttt}(s) || ds \Big) \nonumber\\
& \le C \Delta t^3 \max \limits_{0 \le t \le T} || u_{tttt}(t) ||  \le \frac{C \varepsilon^2 \Delta t^3}{\min\{\varepsilon^6, \delta^6\}}
\end{align}
and hence $\Delta t \sum_{j = 2}^n || z^j_2 - z^{j - 1}_2 || \le  \frac{C \varepsilon^2 \Delta t^3}{\min\{\varepsilon^6, \delta^6\}}$. Therefore
\begin{align}
|| \theta^N ||_e^2 \le || \theta^0 ||_e^2 + C \varepsilon \left( \frac{\Delta t H^2}{ \min\{ \varepsilon^4, \delta^4 \} } + \frac{\varepsilon^2 \Delta t^3}{\min\{\varepsilon^6, \delta^6\}} \right) \sum_{n = 1}^N ( || z^n_1 || + || z^n_2 || ).
\end{align}
According to the proof of Theorem \ref{thm: L2 error in Crank-Nicolson}, we have
$\Delta t \sum_{n = 1}^N || z^n_1 || \le \frac{C H^2}{\varepsilon \min\{ \varepsilon^2, \delta^2 \}}$ and  $\Delta t \sum_{n = 1}^N || z^n_2 || \le \frac{C \varepsilon \Delta t^2}{\min\{ \varepsilon^4, \delta^4 \}}$. Therefore, we obtain
\begin{align}
|| U^N - u(T) ||_e \le C \left( \frac{H}{\varepsilon} + \frac{H^2}{\min\{ \varepsilon^3, \delta^3 \}} + \frac{\varepsilon^2 \Delta t^2}{\min\{ \varepsilon^5, \delta^5 \}} \right).
\end{align}
Moreover, if $u_t(t), u_{tt}(t), u_{ttt}(t) \in H^1(\Omega)$ for any $t \in [0, T]$, then we have $|| \rho^N ||_e \le C \frac{H^2}{\varepsilon^2 \delta}$ and for $j = 2, 3, \ldots, N$,
\begin{align}
|| z^j_1 - z^{j - 1}_1 || \le C \frac{\Delta t H^3}{\varepsilon} \max \limits_{0 \le t \le T} || u_{ttt}(t) ||_{H^1} \le \frac{C \Delta t H^3}{\varepsilon \delta \min\{ \varepsilon^4, \delta^4 \} }.
\end{align}
And from the proof of Theorem \ref{thm: L2 error in Crank-Nicolson}, we know that
\begin{align}
|| \theta^1 - \theta^0 || \le C \left( \frac{\varepsilon \Delta t^3}{\min\{ \varepsilon^4, \delta^4 \}} + \frac{\Delta t H^3}{\varepsilon^2 \delta \min\{ \varepsilon^2, \delta^2 \}} \right), \quad \Delta t \sum_{n = 1}^N || z^n_1 || \le \frac{C H^3}{\varepsilon^2 \delta \min\{ \varepsilon^2, \delta^2 \}}.
\end{align}
Hence we can obtain the estimate \eqref{eq: super convergence of energy error in Crank Nicolson} using similar arguments.
\end{proof}

\begin{rem}
For the estimates in Theorem \ref{thm: L2 error in Crank-Nicolson}, and Theorem \ref{thm: energy error in Crank-Nicolson}, the constants $C$'s depend polynomially and at most quadratically on the final time $T$.
\end{rem}

\section{Numerical experiments}\label{sec: numerical experiments}

In this section, we present numerical results to justify our analysis, where the potential is smooth in one example and possesses discontinuities in the other. We consider \eqref{eq: time-dependent Schrodinger equation} in one dimension with domain $\Omega = [0, 2\pi]$, finial time $T = 0.5$ and initial data
\begin{equation}
u_0(x) = \left( \frac{10}{\pi} \right)^{1/4} \exp \left( - 5 (x - \pi)^2 \right) \exp \left( - i \frac{(x - \pi)^2}{\varepsilon} \right).
\end{equation}
We will compare the relative errors between the numerical solution $u_{\text{num}}$ and the reference solution $u_{\text{ref}}$ in $L^2$ norm and $H^1$ norm with
\begin{align}
\text{err}_{L^2} = \frac{|| u_{\text{num}} - u_{\text{ref}} ||}{|| u_{\text{ref}} ||}, \quad \text{err}_{H^1} = \frac{|| u_{\text{num}} - u_{\text{ref}} ||_{H^1}}{|| u_{\text{ref}} ||_{H^1}}.
\end{align}
Recall that the $H^1$ norm is equivalent to the energy norm.

\subsection{Smooth potentials}

Consider the smooth potential
\begin{equation} \label{eq: smooth potential}
V = \cos\left( \frac{x}{\delta} \right) + 2.
\end{equation}
We choose (i) $\varepsilon = \frac{1}{8}, \delta = \frac{1}{10}$ and (ii) $\varepsilon = \frac{1}{32}, \delta = \frac{1}{24}$. The reference solution is computed by the time-splitting spectral method \cite{bao2002time} with $\Delta t = \frac{1}{2^{26}}, H = \frac{\pi}{2^{15}}$. The numerical solution is computed by the Crank-Nicolson standard linear FEM and the Crank-Nicolson localized OC MsFEM with $\Delta t = \frac{1}{2^{24}}$ and $H = \frac{\pi}{64}, \frac{\pi}{96}, \frac{\pi}{128}, \frac{\pi}{192}, \frac{\pi}{256}$ for case (i), $H = \frac{\pi}{96}, \frac{\pi}{128}, \frac{\pi}{192}, \frac{\pi}{256}, \frac{\pi}{384}$ for case (ii). The oversampling size for the localized OC MsFEM is chosen as $m = 3 \lceil \log_2(\frac{2\pi}{H}) \rceil$. The results are shown in Table \ref{tab: smooth potential eps 8}, Table \ref{tab: smooth potential eps 32}, Figure  \ref{fig: smooth potential eps 8}, and Figure \ref{fig: smooth potential eps 32}.

\begin{table}[h]
\caption{Errors for potential \eqref{eq: smooth potential} with $\varepsilon = 1/8$ and $\delta = 1/10$.} \label{tab: smooth potential eps 8}
\begin{center}
\begin{tabular}{|c c c c c c|}
\hline
$H$ & $\frac{\pi}{64}$ & $\frac{\pi}{96}$ & $\frac{\pi}{128}$ & $\frac{\pi}{192}$ & $\frac{\pi}{256}$ \\
\hline
$\text{err}_{L^2}$ of FEM & 1.0609E-01 & 4.9109E-02 & 2.8067E-02 & 1.2637E-02 & 7.1021E-03 \\
convergence order & & 1.90 & 1.94 & 1.83 & 2.24 \\
$\text{err}_{L^2}$ of localized OC MsFEM & 2.7487E-04 & 3.8263E-05 & 1.1229E-05 & 2.1894E-06 & 8.0399E-07 \\
convergence order & & 4.86 & 4.26 & 3.75 & 3.90 \\
\hline
$\text{err}_{H^1}$ of FEM & 2.7651E-01 & 1.4938E-01 & 9.9142E-02 & 5.8818E-02 & 4.1849E-02 \\
convergence order & & 1.52 & 1.43 & 1.20 & 1.32 \\
$\text{err}_{H^1}$ of localized OC MsFEM & 3.6524E-03 & 9.8017E-04 & 4.0261E-04 & 1.1748E-04 & 4.8910E-05 \\
convergence order & & 3.24 & 3.09 & 2.82 & 3.41 \\
\hline
\end{tabular}
\end{center}
\end{table}

\begin{table}[h]
\caption{Errors for potential \eqref{eq: smooth potential} with $\varepsilon = 1/32$ and $ \delta = 1/24$.} \label{tab: smooth potential eps 32}
\begin{center}
\begin{tabular}{|c c c c c c|}
\hline
$H$ & $\frac{\pi}{96}$ & $\frac{\pi}{128}$ & $\frac{\pi}{192}$ & $\frac{\pi}{256}$ & $\frac{\pi}{384}$ \\
\hline
$\text{err}_{L^2}$  of FEM & 1.0269E+00 & 7.6698E-01 & 4.4394E-01 & 2.7804E-01 & 1.3189E-01 \\
convergence order & & 1.01 & 1.25 & 1.82 & 1.71 \\
$\text{err}_{L^2}$  of localized OC MsFEM & 1.3898E-01 & 2.7714E-02 & 1.8745E-03 & 3.2749E-04 & 4.9916E-05 \\
convergence order & & 5.60  & 6.17 & 6.79 & 4.31 \\
\hline
$\text{err}_{H^1}$ of FEM & 1.3217E+00 & 1.0963E+00 & 7.1925E-01 & 4.9949E-01 & 2.6246E-01 \\
convergence order & & 0.65 & 0.97 & 1.42 & 1.48 \\
$\text{err}_{H^1}$ of localized OC MsFEM & 3.0094E-01 & 7.7003E-02 & 1.0450E-02 & 3.6681E-03 & 9.8675E-04 \\
convergence order & & 4.47 & 4.58 & 4.08 & 3.01 \\
\hline
\end{tabular}
\end{center}
\end{table}

\begin{figure}[H]
	\centering
	\begin{subfigure}{0.4\textwidth}
		\includegraphics[width=\linewidth]{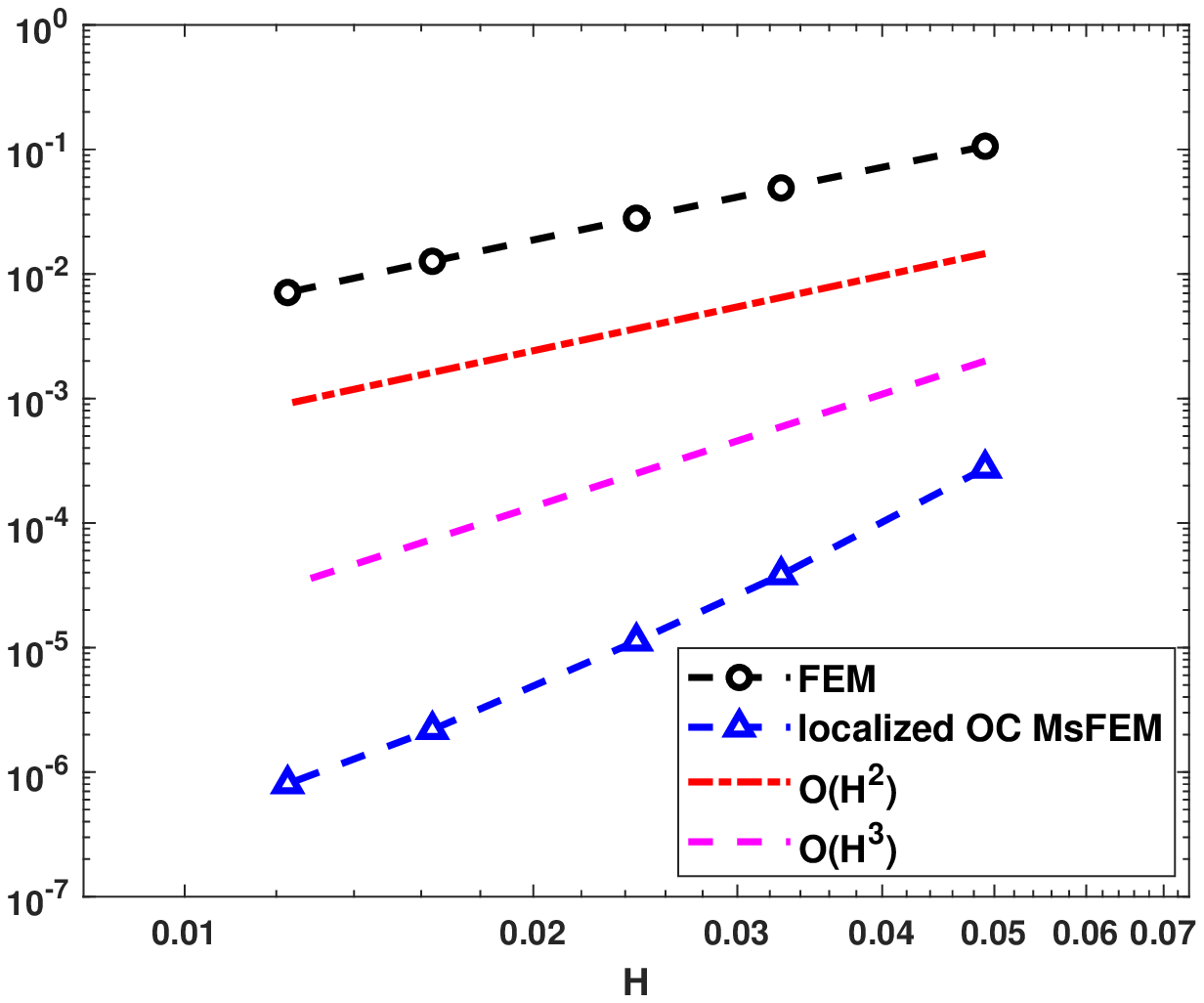}\par
		\caption{$L^2$ relative error $\text{err}_{L^2}$}
	\end{subfigure}
	\begin{subfigure}{0.4\textwidth}
		\includegraphics[width=\linewidth]{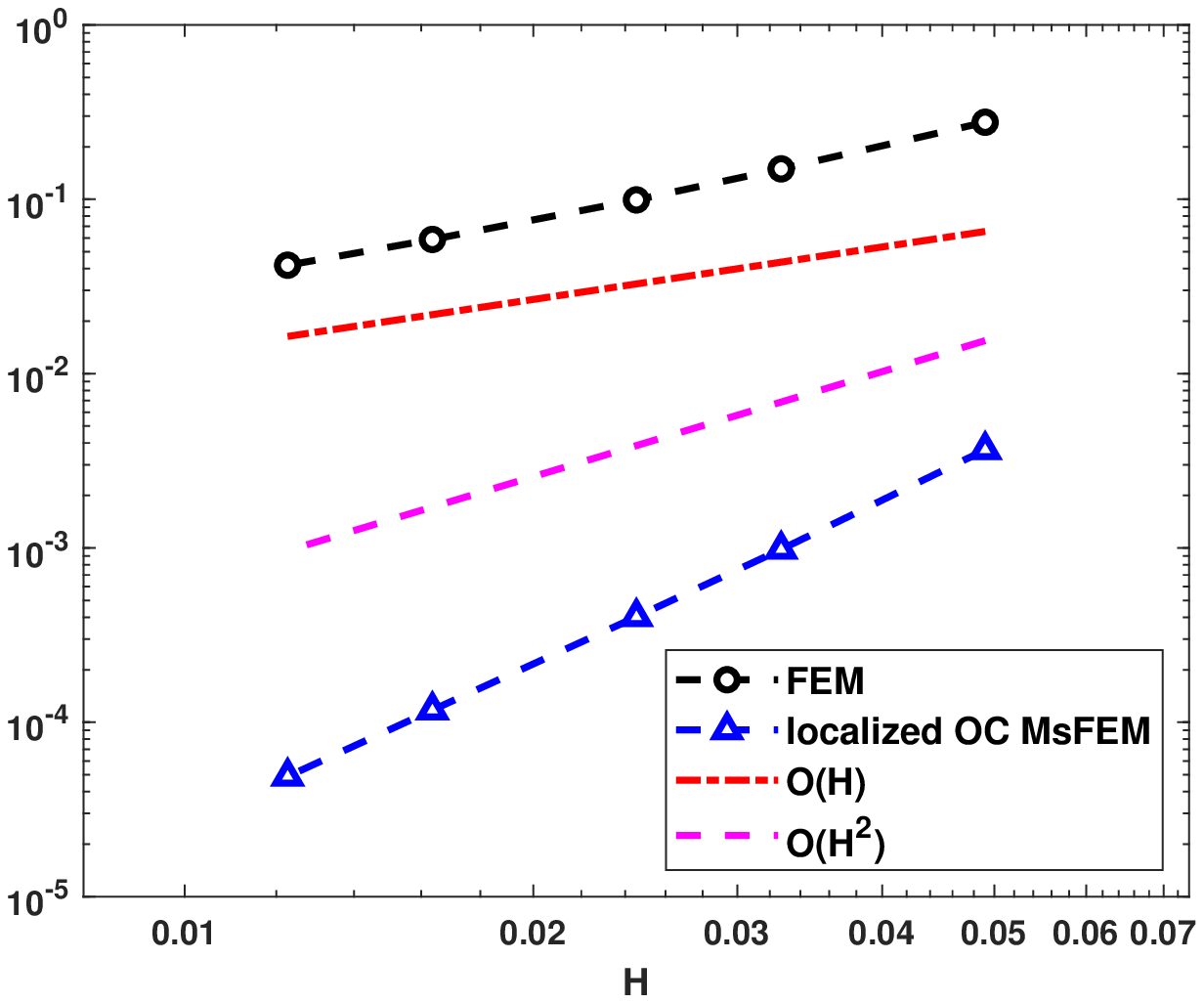}\par
		\caption{$H^1$ relative error $\text{err}_{H^1}$}
	\end{subfigure}
	\caption{Errors for potential \eqref{eq: smooth potential} with $\varepsilon = 1/8$ and $ \delta = 1/10$.}
	\label{fig: smooth potential eps 8}
\end{figure}

\begin{figure}[H]
	\centering
	\begin{subfigure}{0.4\textwidth}
		\includegraphics[width=\linewidth]{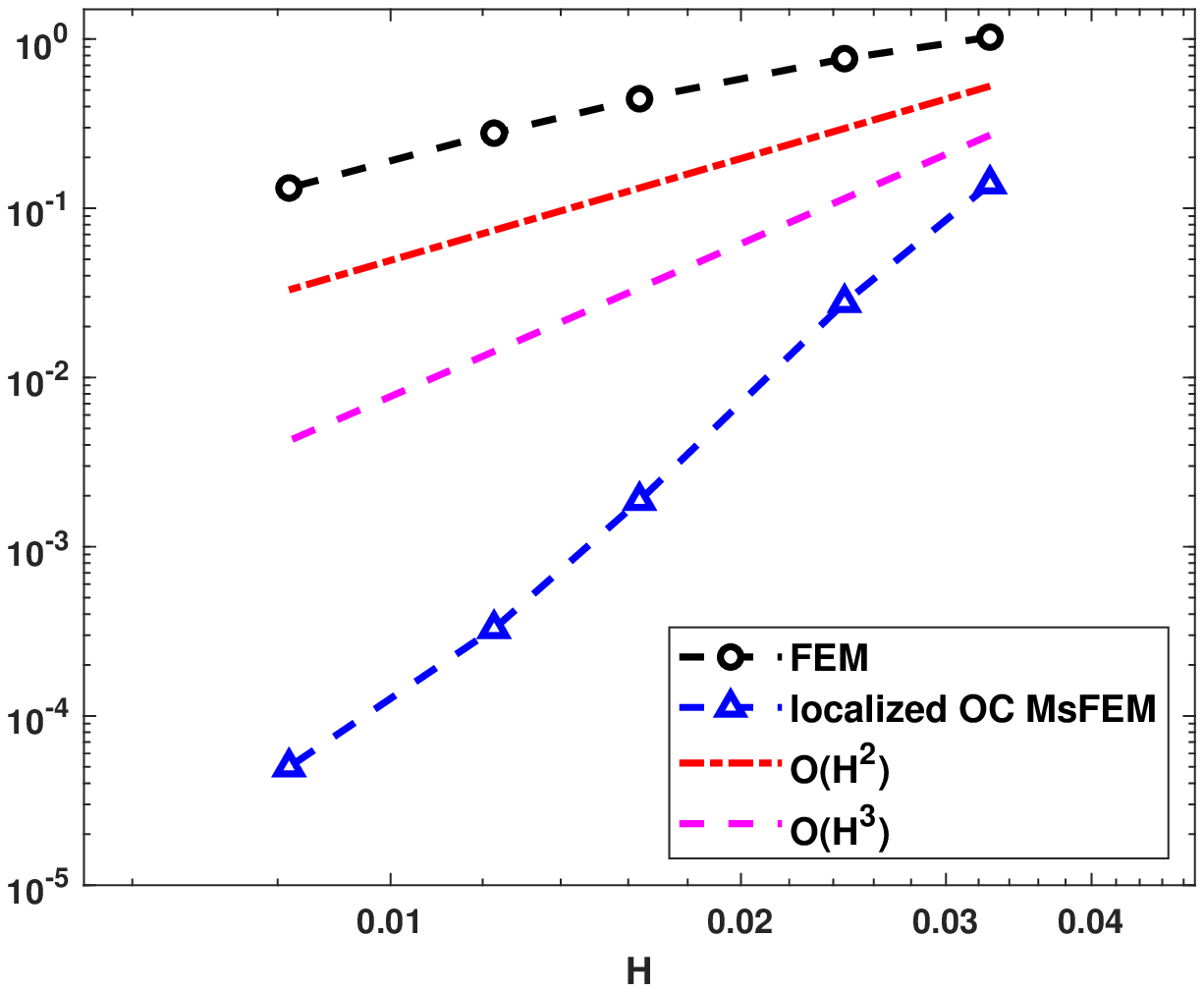}\par
		\caption{$L^2$ relative error $\text{err}_{L^2}$}
	\end{subfigure}
	\begin{subfigure}{0.4\textwidth}
		\includegraphics[width=\linewidth]{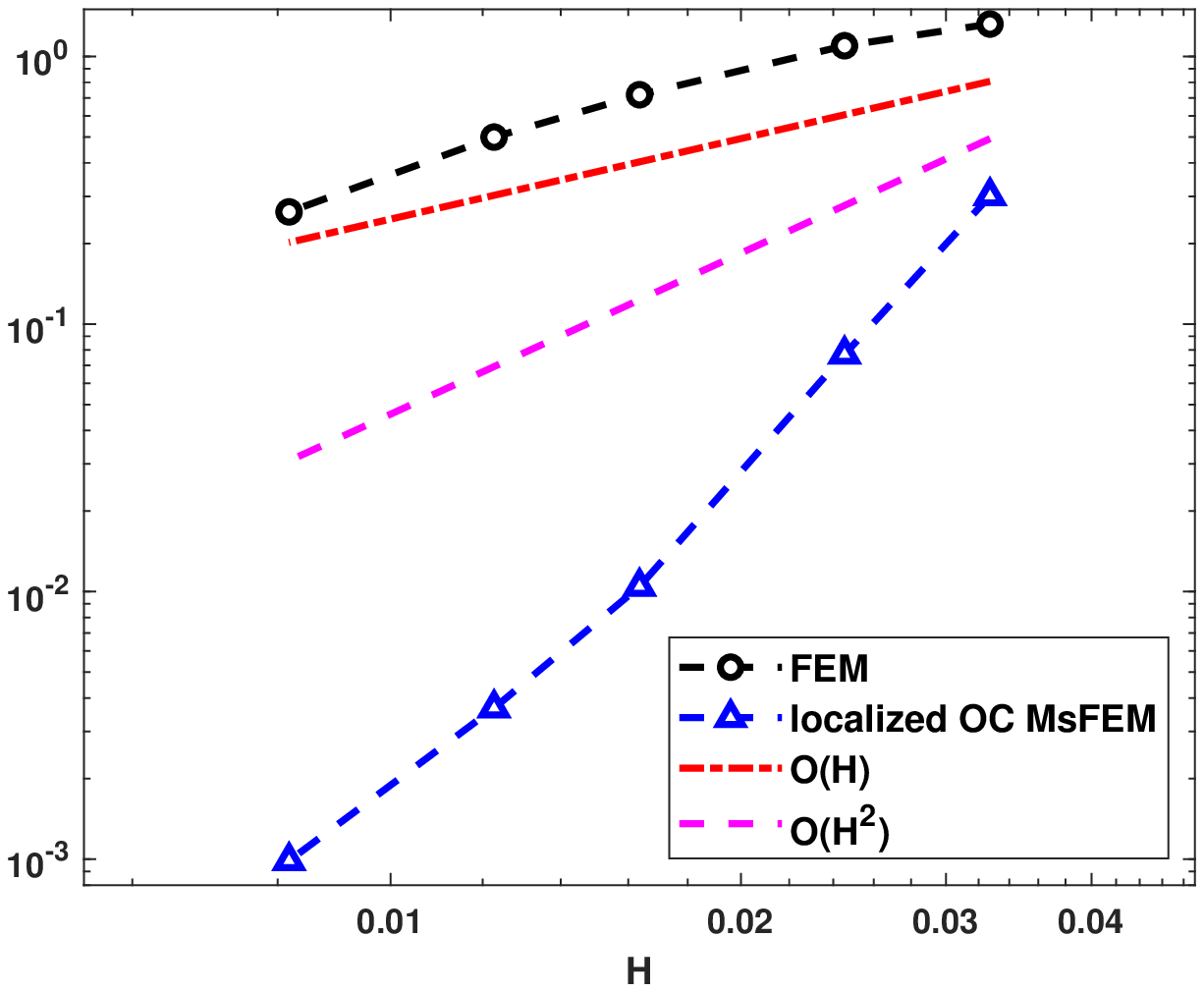}\par
		\caption{$H^1$ relative error $\text{err}_{H^1}$}
	\end{subfigure}
	\caption{Errors for potential \eqref{eq: smooth potential} with $\varepsilon = 1/32$ and $ \delta = 1/24$.}
	\label{fig: smooth potential eps 32}
\end{figure}

For the standard linear FEM, first-order convergence in the energy norm and second-order convergence in the $L^2$ norm are observed. While for the localized OC MsFEM, super convergence is observed and the convergence rates are even higher than the estimates \eqref{eq: super convergence of L2 error in Crank Nicolson}, \eqref{eq: super convergence of energy error in Crank Nicolson} proposed in Theorem \ref{thm: L2 error in Crank-Nicolson} and Theorem \ref{thm: energy error in Crank-Nicolson}. This super-convergence behavior is due to the smoothness of the potential \eqref{eq: smooth potential} that results in a solution with high regularity.

\subsection{Discontinuous Potentials}

Consider the potential
\begin{equation} \label{eq: discontinuous potential}
V = | x - \pi |^2 + 2 +
\left\{
\begin{aligned}
& \cos\left( \frac{x}{\delta_1} \right), \quad x \in [0, \pi], \\
& \cos\left( \frac{x}{\delta_2} \right), \quad x \in (\pi, 2\pi].
\end{aligned}
\right.
\end{equation}
We choose (i) $\varepsilon = \frac{1}{8}, \delta_1 = \frac{1}{5}, \delta_2 = \frac{1}{10}$ and (ii) $\varepsilon = \frac{1}{32}, \delta_1 = \frac{1}{40}, \delta_2 = \frac{1}{25}$. In both cases, the potential \eqref{eq: discontinuous potential} is discontinuous at $x = \pi$ and has different lattice structures on $[0, \pi]$ and $(\pi, 2\pi]$. The reference solution is computed by the Crank-Nicolson global OC MsFEM with $\Delta t = \frac{1}{2^{26}}, H = \frac{\pi}{1024}$. The numerical solutions are computed by the time-splitting spectral method (TSSP), Crank-Nicolson standard linear FEM and Crank-Nicolson localized OC MsFEM with $\Delta t = \frac{1}{2^{24}}$ and $H = \frac{\pi}{64}, \frac{\pi}{96}, \frac{\pi}{128}, \frac{\pi}{192}, \frac{\pi}{256}$ for case (i), $H = \frac{\pi}{96}, \frac{\pi}{128}, \frac{\pi}{192}, \frac{\pi}{256}, \frac{\pi}{384}$ for case (ii). The oversampling size for the localized OC MsFEM is chosen as $m = 2 \lceil \log_2(\frac{2\pi}{H}) \rceil$. The results are shown in Table  \ref{tab: discontinuous potential eps 8}, Table \ref{tab: discontinuous potential eps 32}, Figure \ref{fig: discontinuous potential eps 8}, and Figure \ref{fig: discontinuous potential eps 32}.

\begin{table}[H]
\caption{Errors for potential \eqref{eq: discontinuous potential} with $\varepsilon = 1/8$, $\delta_1 = 1/5$, and $\delta_2 = 1/10$.} \label{tab: discontinuous potential eps 8}
\begin{center}
\begin{tabular}{|c c c c c c|}
\hline
$H$ & $\frac{\pi}{64}$ & $\frac{\pi}{96}$ & $\frac{\pi}{128}$ & $\frac{\pi}{192}$ & $\frac{\pi}{256}$ \\
\hline
$\text{err}_{L^2}$ of TSSP & 4.1036E-01 & 2.6171E-01 & 2.0189E-01 & 1.1358E-01 & 9.9578E-02 \\
convergence order & & 1.11 & 0.90 & 1.32 & 0.51 \\
$\text{err}_{L^2}$ of FEM & 1.2790E-01 & 6.8296E-02 & 4.4099E-02 & 2.4116E-02 & 1.5917E-02 \\
convergence order & & 1.55 & 1.52 & 1.38 & 1.62 \\
$\text{err}_{L^2}$ of localized OC MsFEM & 8.8626E-03 & 4.0528E-03 & 2.3565E-03 & 1.1604E-03 & 6.4039E-04 \\
convergence order & & 1.93 & 1.88 & 1.62 & 2.31 \\
\hline
$\text{err}_{H^1}$ of TSSP & 5.1651E-01 & 3.2669E-01 & 2.5942E-01 & 1.4208E-01 & 1.2934E-01 \\
convergence order & & 1.13 & 0.80 & 1.38 & 0.37 \\
$\text{err}_{H^1}$ of FEM & 3.3974E-01 & 2.2411E-01 & 1.7045E-01 & 1.1976E-01 & 9.5007E-02 \\
convergence order & & 1.03 & 0.95 & 0.81 & 0.90 \\
$\text{err}_{H^1}$ of localized OC MsFEM & 5.7921E-02 & 4.3351E-02 & 3.2635E-02 & 2.2383E-02 & 1.5132E-02 \\
convergence order & & 0.71 & 0.99 & 0.86 & 1.52 \\
\hline
\end{tabular}
\end{center}
\end{table}

\begin{table}[H]
\caption{Errors for potential \eqref{eq: discontinuous potential} with $\varepsilon = 1/32$, $\delta_1 = 1/40$, and $\delta_2 = 1/25$.} \label{tab: discontinuous potential eps 32}
\begin{center}
\begin{tabular}{|c c c c c c|}
\hline
$H$ & $\frac{\pi}{96}$ & $\frac{\pi}{128}$ & $\frac{\pi}{192}$ & $\frac{\pi}{256}$ & $\frac{\pi}{384}$ \\
\hline
$\text{err}_{L^2}$ of TSSP & 8.6933E-01 & 8.2000E-01 & 7.2275E-01 & 5.1762E-01 & 4.7231E-01 \\
convergence order & & 0.20 & 0.29 & 1.30 & 0.21 \\
$\text{err}_{L^2}$ of FEM & 1.0980E+00 & 8.2753E-01 & 4.9740E-01 & 3.2823E-01 & 1.6666E-01 \\
convergence order & & 0.98 & 1.17 & 1.62 & 1.55 \\
$\text{err}_{L^2}$ of localized OC MsFEM & 1.8558E-01 & 5.9763E-02 & 1.4948E-02 & 7.0768E-03 & 2.7845E-03 \\
convergence order & & 3.94 & 3.18 & 2.91 & 2.14 \\
\hline
$\text{err}_{H^1}$ of TSSP & 9.4843E-01 & 8.6754E-01 & 6.5087E-01 & 5.2097E-01 & 6.0281E-01 \\
convergence order & & 0.31 & 0.66 & 0.87 & -0.33 \\
$\text{err}_{H^1}$ of FEM & 1.1400E+00 & 9.5249E-01 & 6.7419E-01 & 4.8814E-01 & 2.9260E-01 \\
convergence order & & 0.62 & 0.79 & 1.26 & 1.17 \\
$\text{err}_{H^1}$ of localized OC MsFEM & 3.1450E-01 & 1.3862E-01 & 5.3020E-02 & 3.2920E-02 & 2.2157E-02 \\
convergence order & & 2.85 & 2.20 & 1.85 & 0.91 \\
\hline
\end{tabular}
\end{center}
\end{table}

\begin{figure}[H]
	\centering
	\begin{subfigure}{0.4\textwidth}
		\includegraphics[width=\linewidth]{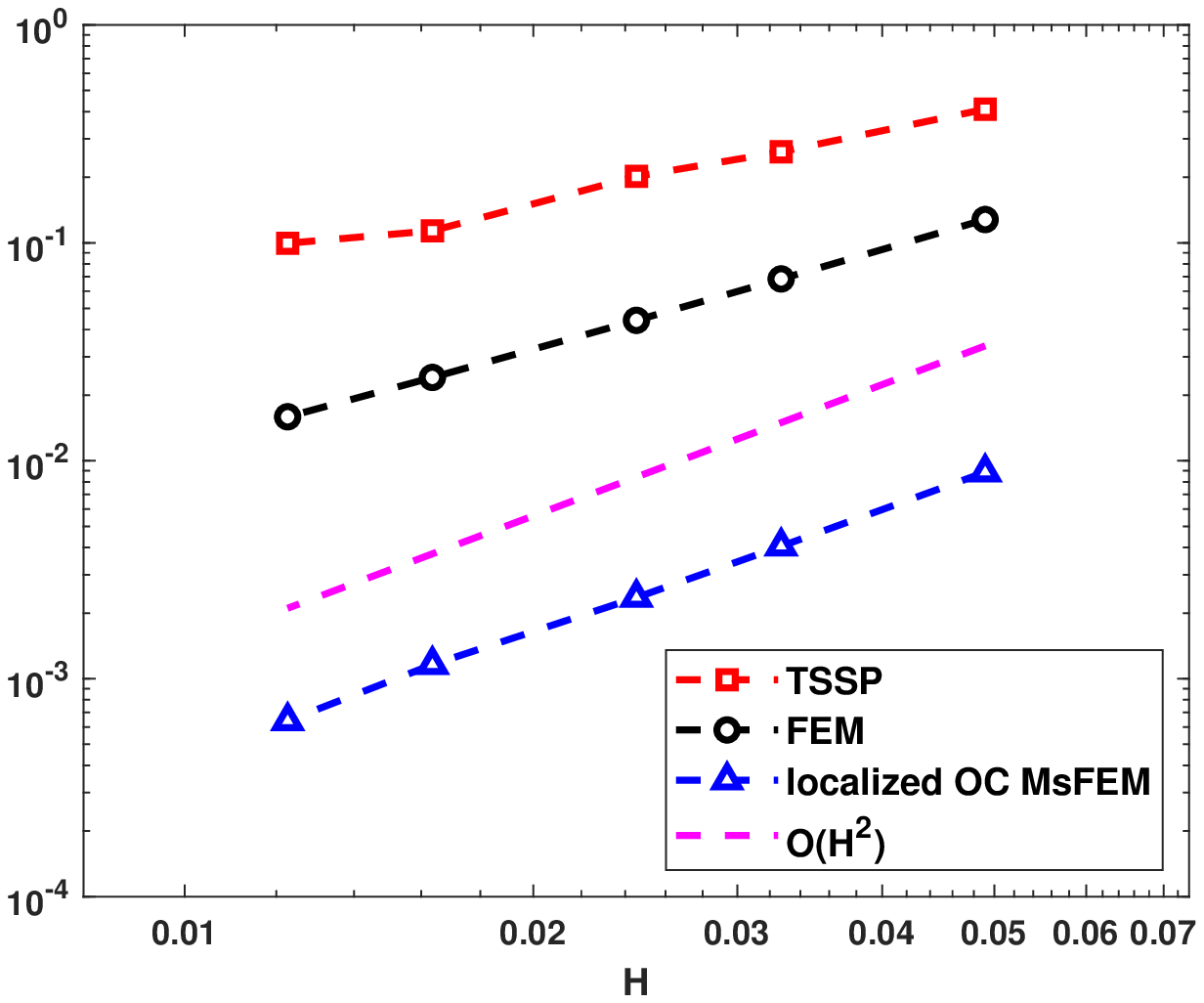}\par
		\caption{$L^2$ relative error $\text{err}_{L^2}$}
	\end{subfigure}
	\begin{subfigure}{0.4\textwidth}
		\includegraphics[width=\linewidth]{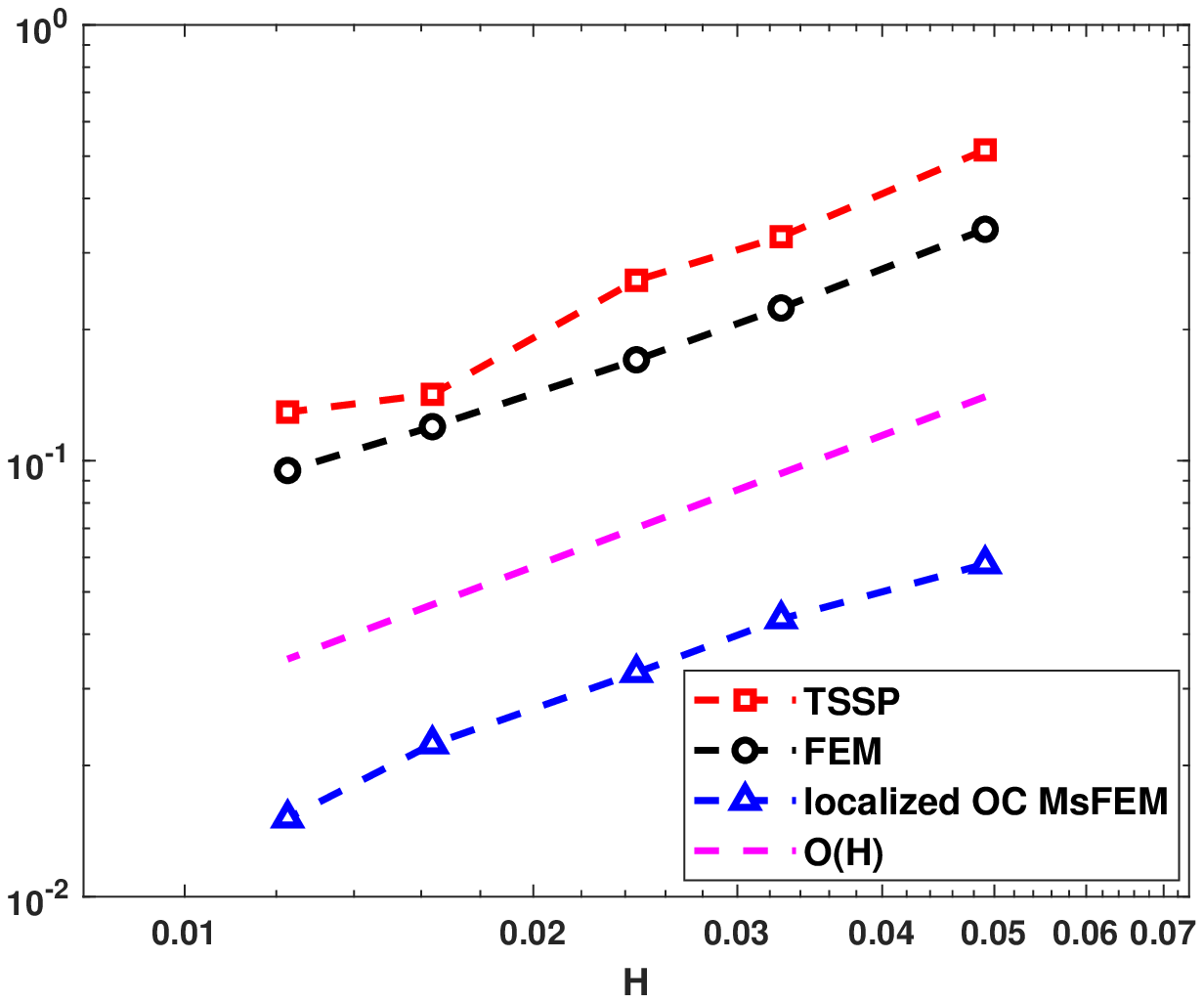}\par
		\caption{$H^1$ relative error $\text{err}_{H^1}$}
	\end{subfigure}
	\caption{Errors for potential \eqref{eq: discontinuous potential} with $\varepsilon = 1/8$, $\delta_1 = 1/5$, and $\delta_2 = 1/10$.}
	\label{fig: discontinuous potential eps 8}
\end{figure}

\begin{figure}[H]
	\centering
	\begin{subfigure}{0.4\textwidth}
		\includegraphics[width=\linewidth]{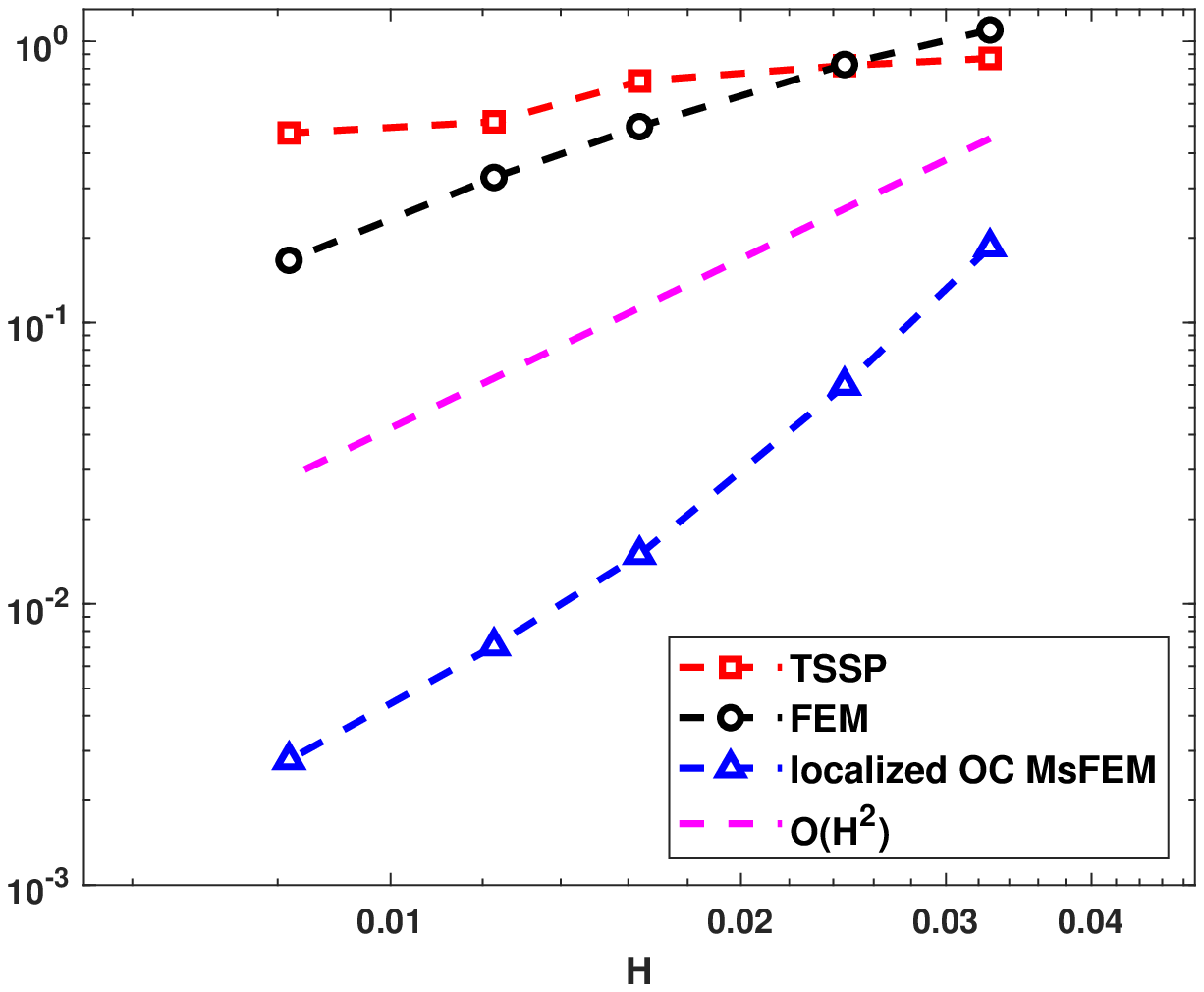}\par
		\caption{$L^2$ relative error $\text{err}_{L^2}$}
	\end{subfigure}
	\begin{subfigure}{0.4\textwidth}
		\includegraphics[width=\linewidth]{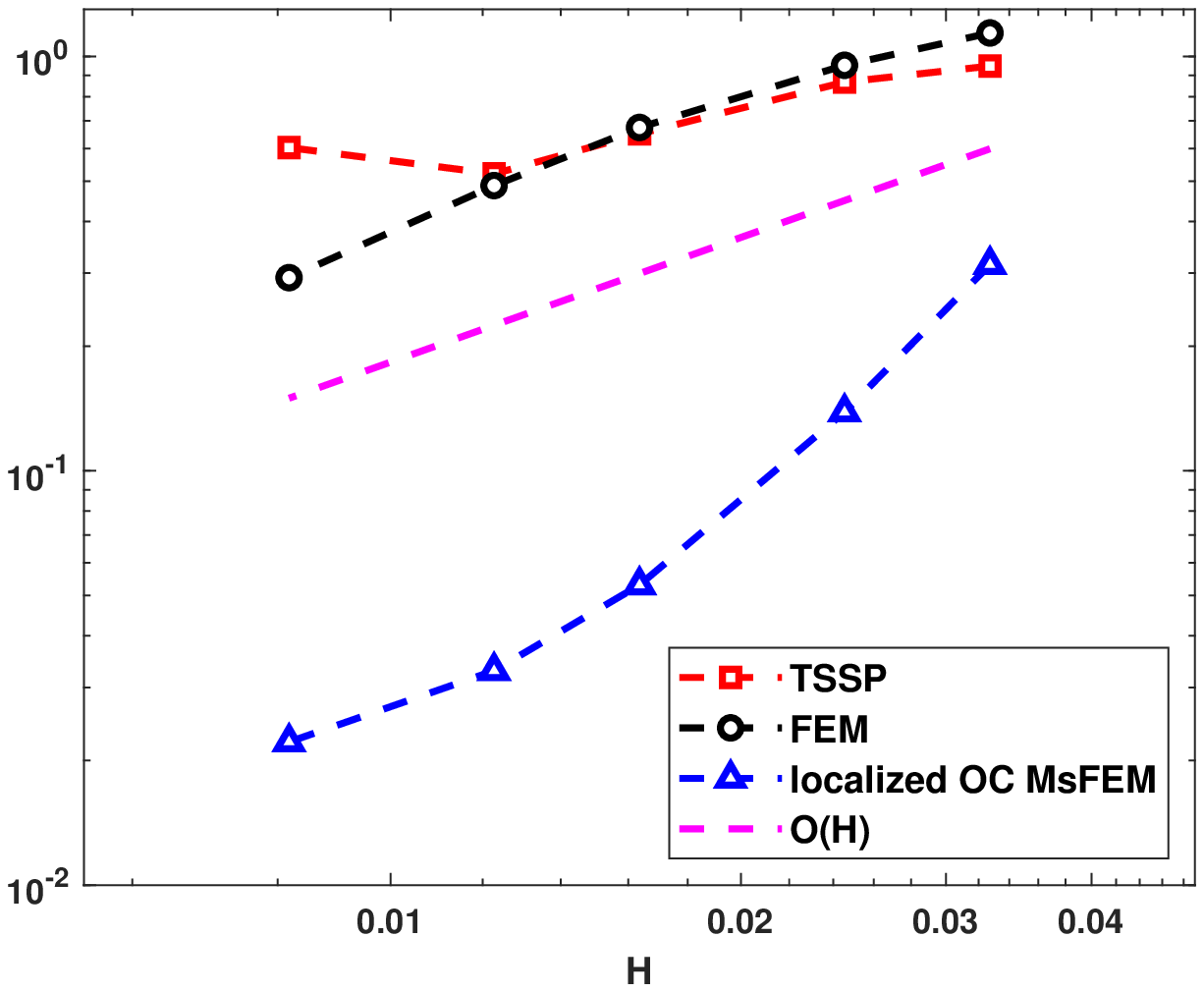}\par
		\caption{$H^1$ relative error $\text{err}_{H^1}$}
	\end{subfigure}
	\caption{Errors for potential \eqref{eq: discontinuous potential} with $\varepsilon = 1/32$, $\delta_1 = 1/40$, and $\delta_2 = 1/25$.}
	\label{fig: discontinuous potential eps 32}
\end{figure}

The time-splitting spectral method suffers from reduced convergence order and low accuracy due to the discontinuous potential \eqref{eq: discontinuous potential}. However, convergence rates of approximately first order in the energy norm and second order in the $L^2$ norm are still observed for the FEM and OC MsFEM although the discontinuous potential \eqref{eq: discontinuous potential} results in a solution with lower regularity. Moreover, the OC MsFEM yields much higher accuracy than the FEM, which is due to the fact that for a solution with low regularity in the presence of a multiscale potential, the projection error estimates of the OC MsFEM depend more weakly on the small parameters than the FEM. The result is consistent with the discussion in Section \ref{sec: Projection error}. This numerical example shows that the OC MsFEM is robust in the sense that it still yields an accurate solution for the semiclassical Schr\"{o}dinger equation with a discontinuous multiscale potential.

Both examples confirm our theoretical findings and indicate that the OC MsFEM is accurate and robust for the Schr\"{o}dinger equation with  general multiscale potentials.

\section{Conclusion} \label{sec: conclusion}
In this paper, we provide a rigorous convergence analysis  for the OC MsFEM in solving Schr\"{o}dinger equations with general multiscale potentials in the semiclassical regime. We prove the exponential decay of the multiscale basis functions and propose the way of constructing the localized multiscale basis functions. Besides, we show that the localized basis functions can achieve the same accuracy as the global ones by choosing the oversampling size $m$ appropriately according to the mesh size $H$ as $m = O(\log(1/H))$. Based on the properties of Cl\'{e}ment-type interpolation, we prove that the OC MsFEM can achieve first-order convergence in energy norm and second-order convergence in $L^2$ norm. Furthermore, if the solution possesses sufficiently high regularity, super convergence rates of second order in energy norm and third order in $L^2$ norm can be obtained. By combining the analysis on the regularity of the solution, we also derive the dependence of the error estimates on the small parameters. We find that using the same mesh size the OC MsFEM gives more accurate results than the FEM in solving Schr\"{o}dinger equations with multiscale potentials due to its super convergence behavior for high-regularity solutions and weaker dependence on the small parameters $\varepsilon$ and $\delta$ for low-regularity solutions. The weaker dependence of the OC MsFEM on the small parameters is a consequence of the fact that the time derivatives of the solution are less oscillatory than the spatial derivatives in the presence of the multiscale potential. Numerical results confirm our analysis. For a smooth potential, super convergence rates are observed for the OC MsFEM. While for a discontinuous potential, the OC MsFEM retains first-order and second-order convergence in the energy norm and $L^2$ norm respectively and still yields high accuracy. Therefore, the OC MsFEM is accurate and robust for the semiclassical Schr\"{o}dinger equation with various types of multiscale potentials.

In the future, we will study the convergence analysis of the OC MsFEM for solving eigenvalue problems for the Schr\"{o}dinger operators and  nonlinear Schr\"{o}dinger equations in the semiclassical regime. In addition, we will apply the OC MsFEM to solve wave equations with multiscale features, such as the Klein-Gordon equation \cite{huang2009bloch}.

\section*{Acknowledgement}
\noindent
The research of Z. Zhang is supported by Hong Kong RGC grants (Projects 17300817, 17300318, and 17307921), Seed Funding Programme for Basic Research (HKU), and Basic Research Programme (JCYJ20180307151603959) of The Science, Technology and Innovation Commission of Shenzhen Municipality. The computations were performed using research computing facilities offered by Information Technology Services, the University of Hong Kong.

\bibliographystyle{siam}
\bibliography{reference}

\end{document}